\documentclass[11pt,a4paper,reqno]{amsart}
\usepackage[utf8]{inputenc}
\usepackage{color}
\usepackage{amsmath}
\usepackage{amssymb}
\usepackage{bm}
\usepackage{bbm}
\usepackage{hyperref}
\usepackage{amsthm}
\usepackage{geometry}

\usepackage[english]{babel}

\usepackage[
backend=biber,
style=alphabetic,
sorting=ynt,
giveninits=true
]{biblatex}

\bibliography{papers.bib}

\def\D{\mathcal{D}}
\def\F{\mathcal{F}}
\def\G{\mathcal{G}}
\def\H{\mathcal{H}}
\def\E{\mathcal{E}}
\def\Rc{\mathcal{R}}
\def\R{\mathbb{R}}
\def\Pc{\mathcal{P}}
\def\e{\varepsilon}
\def\d{\mathsf{d}}
\def\h{\mathsf{h}}
\def\Hm{\mathsf{H}}

\def\M{\mathcal{M}}
\def\div{\mathrm{div}}

\newcommand{\mres}{\mathbin{\vrule height 1.4ex depth 0pt width
0.13ex\vrule height 0.13ex depth 0pt width 1.4ex}}

\newcommand{\cb}{\color{blue}}

\newtheorem{theorem}{Theorem}[section]
\newtheorem{lemma}[theorem]{Lemma}
\newtheorem{prop}[theorem]{Proposition}
\newtheorem{cor}[theorem]{Corollary}
\theoremstyle{remark}
\newtheorem{remark}[theorem]{Remark}

\title[Entropic regularisation of UOT]{Entropic regularisation of unbalanced optimal transportation problems}
\author{M. Buze}
\address{School of Mathematical and Computer Sciences, Heriot-Watt University, Edinburgh, Scotland, EH14 4AS, United Kingdom}
\email{m.buze@hw.ac.uk}

\author{M. H. Duong}
\address{School of Mathematics, Watson Building, University of Birmingham, Edgbaston, Birmingham, B15 2TT UK}
\email{h.duong@bham.ac.uk}

\thanks{Both authors would like acknowledge the support of the Engineering and Physical Sciences Research Council in the UK. MHD is supported under research grant EP/V038516/1.  MB was supported under the same research grant for most of the work towards this paper and is now supported by research grant EP/V00204X/1.}

\subjclass[2020]{49Q22, 28A33, 46E27, 58E30, 90C25, 49N05}

\date{\today}

\keywords{entropic optimal transport, unbalanced optimal transport, regularization}

\begin{document}

\maketitle

\begin{abstract}
    We develop a mathematical theory of entropic regularisation of unbalanced optimal transport problems. Focusing on static formulation and relying on the formalism developed for the unregularised case, we show that unbalanced optimal transport problems can be regularised in two qualitatively distinct ways -- either \emph{on the original space} or \emph{on the extended space}. We derive several reformulations of the two regularised problems and in particular introduce the idea of a \emph{regularised induced marginal perspective cost function} allowing us to derive an \emph{extended space formulation of the original space regularisation}. We also prove convergence to the unregularised problem in the case of the extended space regularisation and discuss on-going work on deriving a unified framework based on \emph{higher order liftings} in which both regularisations can be directly compared. We also briefly touch upon how these concepts translate to the corresponding dynamic formulations and provide evidence why the extended space regularisation should be preferred. This is a preliminary version of the manuscript, to be updated in the near future.
\end{abstract}

\section{Introduction}

Given some complete separable metric space $X$, two positive Radon measures with finite mass,  $\mu_0,\mu_1 \in \M(X)$, and a lower semicontinuous cost function $c\,\colon\, X \times X \to \overline\R$, the standard \emph{balanced} optimal transport problem is given by 
\[
{\rm OT}(\mu_0,\mu_1) = \inf_{\gamma \in \Gamma(\mu_0,\mu_1)}(c,\gamma), \quad (c,\gamma):= \int_{X\times X}c(x_0,x_1)d\gamma(x_0,x_1),
\]
\[
\Gamma(\mu_0,\mu_1) := \Big\{ \gamma \in \M(X \times X) \mid \gamma_i = \mu_i\Big\},
\]
where $\gamma_i$ is the $i$th marginal, 
\[
\gamma_i := \pi^i_{\#}\gamma \in \M(X), \quad \pi^i(x_0,x_1) = x_i.
\]

In order to allow for efficient algorithms, but also out of theoretical interest, the balanced optimal transport problem is often \emph{entropy-regularised} \cite{N21}, namely 
\[
{\rm OT}_{\e}(\mu_0,\mu_1) = \inf_{\gamma \in \Gamma(\mu_0,\mu_1)}\Big\{ (c,\gamma) + \e \F(\gamma \mid \nu)\Big\},
\]
where $\nu \in \M(X \times X)$ is some known reference measure (e.g. $\nu = \mu_0 \otimes \mu_1$), the entropy functional $\F$ is given by 
\[
\F(\gamma \mid \nu) = \int_{X \times X}F(\varsigma)d\nu + F_{\infty}'\gamma^{\perp}(X\times X), \quad \gamma = \varsigma \nu + \gamma^{\perp},
\]
where $F(s) = s\log s - s +1$ is the Kullback-Leibler divergence (also referred to as relative entropy) and $F'_{\infty}$ is the recession constant which in this case is equal to $+\infty$, thus enforcing that the minimiser satisfies $\gamma \ll \nu$.

It is not hard to see that that both  problems introduced so far are meaningful only if the masses of input measures are \emph{balanced}, that is if $\mu_0(X) = \mu_1(X)$, as otherwise the set $\Gamma(\mu_0,\mu_1)$ is empty.

Relying in part on the formalism developed in \cite{LMS18}, we can equivalently formulate the unregularised problem as 
\[
{\rm OT}(\mu_0,\mu_1) = \inf_{\gamma \in \M(X \times X)} \Big\{ \sum_i \overline{\F}(\gamma_i \mid \mu_i) + (c,\gamma)\Big\},
\]
where $\overline \F(\gamma_i \mid \mu_i) = 0$ if $\gamma_i = \mu_i$ and $+\infty$ otherwise, which in particular encodes the unfeasbility of the problem when masses of the measures differ, as
\[
\mu_0(X) \neq \mu_1(X) \implies {\rm OT}(\mu_0,\mu_1) = +\infty. 
\]

The idea behind the unbalanced optimal transport problems, introduced independently by three groups \cite{CPSV16,KMV16,LMS18} is to relax the sharpness of $\overline \F$ and replace it with an entropy functional merely penalising the deviations between the marginals of $\gamma$ and $\mu_i$, 
\[
\F(\gamma_i \mid \mu_i) = \int_X F(\sigma_i)d\mu_i, \quad \gamma_i = \sigma_i \mu_i + \gamma_i^{\perp},
\]
where the entropy funtion $F$ may not necessarily be the KL divergence introduced above. The unbalanced optimal transport problem is thus given by 
\begin{equation}\label{UOT-intro1}
{\rm UOT}(\mu_0,\mu_1) = \inf_{\gamma \in \M(X \times X)} \Big\{ \sum_i {\F}(\gamma_i \mid \mu_i) + (c,\gamma)\Big\}.
\end{equation}
Remarkably, it can be shown \cite[Section~5]{LMS18} that this problem admits an extended space space formulation
\begin{equation}\label{UOT-intro2}
{\rm UOT}(\mu_0,\mu_1) = \inf_{\alpha \in S(\mu_0,\mu_1)} (H,\alpha), \quad (H,\alpha):= \int_{Y \times Y} H(y_0,y_1)d\alpha(y_0,y_1), 
\end{equation}
where $Y := X \times \R_+$, with notation $y_i = (x_i, s_i)$ and  $H\, \colon\, Y \times Y \to \overline{\R}$ is the induced marginal perspective cost function. The infimum is taken over the set measures satisfying homogeneous marginal contraints
\[
S(\mu_0,\mu_1) = \Big\{ \alpha \in M(Y \times Y) \mid \pi^{x_i}_{\#}(s_i \alpha) = \mu_i.\Big\}
\]
A full account of the concepts introduced so far will be given in Section~\ref{sec:UOT}.\\

The aim of this paper is to study entropy regularisation of the unbalanced optimal transport problems. From the discussion above it is clear that this can be achieved in two ways -- starting from \eqref{UOT-intro1}, we can regularise \emph{on the original space} and consider
\[
{\rm UOT}_{X,\e}(\mu_0,\mu_1) = \inf_{\gamma \in \M(X \times X)} \Big\{ \sum_i {\F}(\gamma_i \mid \mu_i) + (c,\gamma) + \e \F(\gamma \mid \nu_X) \Big\},
\]
for some reference measure $\nu_X \in \M(X \times X)$.

On the other hand, starting from \eqref{UOT-intro2}, we can regularise \emph{on the extended space} and consider 
\[
{\rm UOT}_{Y,\e}(\mu_0,\mu_1) = \inf_{\alpha \in S(\mu_0,\mu_1)}\Big\{ (H,\alpha) + \e \F(\alpha \mid \nu_Y)\Big\},
\]
for some reference measure $\nu_Y \in \M(Y \times Y)$.

\subsection*{Outline of the paper and its contributions} The paper begins with a reminder on the theory of static formulation of unbalanced optimal transport, adapted from \cite{LMS18} and presented in Section~\ref{sec:UOT}. In particular, we highlight how, in order to prove the homogeneous and extended reformulations, one needs to leverage both the dual and reverse formulations. We also show as an example that even the balanced case admits a not necessarily trivial extended description.

The contributions of this paper are described in Sections~\ref{sec:eUOT}--\ref{sec:eUOT-dynamic}. In Section~\ref{sec:eUOT_X} we develop the theory of the original space based regularisation. We closely follow the structure of Section~\ref{sec:UOT} and recover reverse, dual and homogeneous formulations and related results. Interestingly, while the induced marginal perspective cost function in the canonical unregularised case (with KL divergences as entropy functions) is given by
\[
H(x_0,s_0,x_1,s_1) = s_0 + s_1 - 2\sqrt{s_0s_1}\exp\left(\frac{-c(x_0,x_1)}{2}\right),
\]
we show that the corresponding \emph{regularised induced marginal perspective cost function} is given by 
\[
H_\e(x_0,s_0,x_1,s_1,S) = s_0 + s_1 + \e S - (2+\e)\left((s_0s_1)^{\frac{1}{2+\e}} S^{\frac{\e}{2+\e}}\exp\left(\frac{-c(x_0,x_1)}{2+\e}\right)\right),
\]
which is $1$-homogeneous with respect to $(s_0,s_1,S)$. Based on that, we derive extended formulation, where we lift from $\M(X \times X)$ to the space $\M(X^2 \times \R_+^3)$.

In Section~\ref{sec:eUOT_Y} we partially develop the theory of the extended space based regularisation. In particular, we establish that as the regularisation parameter $\e \to 0$ this regularisation converges to the original unbalanced optimal transport problem. To achieve this, we leverage the reformulation of the unbalanced optimal transport problem as an infimum over optimal transport problems on the extended space and the known theory of balanced entropic optimal transport. 

In Section~\ref{sec:eUOT-unify} we discuss the key ingredients of a unified extended framework. On the one hand it will ensure that a direct comparison between the two types of regularisations is possible and thus pave a way for a convergence proof for the original space regularisation. On the other hand the framework establishes the principles of \emph{higher order liftings}, which can be used to formulate \emph{higher order unbalanced optimal transport problems} and its entropic regularisations. 

The paper concludes with Section~\ref{sec:eUOT-dynamic}, in which we sketch out how the key distinction between original- and extended- space regularisations in our static formulation enter in the dynamic formulation and provide evidence that the extended space regularisation appears to lead to asymptotically more rapid convergence as $\e \to 0$.

\subsection*{Comment on the current version of the manuscript} The manuscript in its current form contains several gaps which we aim to fill in the near future. They are clearly indicated throughout the paper with blue-coloured annotations. These gaps are one of two possible kinds: (i) technicalities around results we are convinced to be true; (ii) deeper novel ideas that require significant work and may even be postponed to future papers, but which we felt should be put out there for the community to see. This is also why gaps of the first kind are currently present -- we wanted to publish this manuscript in some form in a timely fashion. In particular, we are aware that our extended space formulations currently gloss over the issues related to sets where radial components are null. We invite people interested in this work to reach out to us to discuss possible mistakes and further steps.  
\section{Unbalanced Optimal Transport}\label{sec:UOT}
We follow the presentation in \cite{LMS18}. Let $(X,\d)$ be a complete and separable metric space and let $\mathcal{M}(X)$ be the space of finite nonnegative Radon measures on $X$. Let $\mu_0,\mu_1 \in \mathcal{M}(X)$. {\color{blue} \footnotesize [In the current version of the manuscript, we implicitly assume ${X = \Omega \subset \R^n}$ is a compact subset of $\R^n$, but keep the discussion as general as possible.]}
\subsection{Static primal formulation}\label{sec:UOT-P}
The class of unbalanced optimal transport problems we consider is given by 
\begin{equation}\label{UOT-P}
{\rm UOT}(\mu_0,\mu_1) := \inf_{\gamma \in \mathcal{M}(X\times X)} \E(\gamma \mid \mu_0,\mu_1),
\end{equation}
where the primal unbalanced optimal transport functional $\E$ is given by 
\begin{equation}\label{P}
\E(\gamma \mid \mu_0, \mu_1) := \sum_{i=0,1} \mathcal{F}(\gamma_i \mid \mu_i) + (c,\gamma).
\end{equation}
Here $\gamma_i = \pi^i_{\#}\gamma \in \M(X)$ denotes the $i$th marginal of $\gamma \in \M(X\times X)$ and the relative entropy functional $\F$ is given by
\begin{equation}\label{F-func}
    \F(\gamma_i \mid \mu_i) := \int_X F(\sigma_i)d\mu_i + F_{\infty}' \gamma_i^{\perp}(X),
\end{equation}
where $F\,\colon [0,+\infty) \to [0,+\infty]$ is a suitable convex entropy function,
\[
{F'_{\infty} = \lim_{s \to \infty} \frac{F(s)}{s}}
\]
is known as the recession constant and $\sigma_i$ is the density obtained via the Lebesgue decomposition \cite[Lemma 2.3]{LMS18}
\begin{equation}\label{L-decomp}
    \gamma_i = \sigma_i \mu_i + \gamma_i^\perp, \quad \mu_i = \rho_i \gamma_i + \mu_i^\perp.
\end{equation}
Finally, $c\,\colon\,X \times X \to \overline\R$ is the cost function, which is assumed to be lower semi-continuous. The coupling
\begin{equation}\label{c-gamma-coupling}
(c,\gamma) := \int_{X\times X} c(x_0,x_1) d\gamma(x_0,x_1)
\end{equation}
denotes the linear cost functional. 
\begin{theorem}[\protect{\cite[Theorem~3.3]{LMS18}}]\label{thm:UOT-P-minimiser}
    Under natural assumptions on the entropy functional $\F$ and the cost function $c$, there exists at least one $\gamma \in \M(X\times X)$ minimizing the right-hand side of \eqref{UOT-P}.
\end{theorem}

Of particular interest to us is the case where the entropy function $F$ is the Kullback-Leibler (KL) divergence
\begin{equation}\label{KL}
    F(s) = s\log s - s  + 1,\quad \Big(\implies F_\infty' = +\infty\;\Big)
\end{equation}
and the cost function $c$ is given by 
\begin{equation}\label{HK-cost-fct}
c(x_0,x_1) := \begin{cases}
-\log(\cos^2(\d(x_0,x_1)))\quad &\text{if }\d(x_0,x_1) < \pi/2 \\
+\infty \quad &\text{otherwise}.
\end{cases}
\end{equation}
In this case the UOT problem \eqref{UOT-P} coincides with the Hellinger-Kantorovich distance on $\M(X)$ \cite[Part~II]{LMS18}. 
\subsection{Equivalent formulations on $X$}\label{sec:UOT-equiv}
The UOT problem in \eqref{UOT-P} admits three equivalent formulations defined on the space $X$, which we will now provide a brief account of. It will turn out in Section~\ref{sec:eUOT_X} that, subject to necessary adjustments, similar results apply to the entropy-regularised UOT problem. 
\subsubsection{Reverse formulation}\label{sec:UOT-R}
Given some entropy function $F$ in \eqref{F-func}, the corresponding reverse entropy function $R\,\colon\, [0,\infty) \to [0,\infty]$ is introduced as
\begin{equation}\label{R-ent}
    R(s) := \begin{cases}
s F(1/s)\quad &\text{if }s>0,\\
F_{\infty}' \quad &\text{if }s=0.
\end{cases}
\end{equation}
Note if $F$ is the KL divergence \eqref{KL} then 
\begin{equation}\label{R-KL}
    R(s) = s - \log s - 1.
\end{equation}
\begin{theorem}[\protect{\cite[Theorem~3.11]{LMS18}}]
The UOT problem \eqref{UOT-P} can be equivalently formulated as 
    \begin{equation*}
        {\rm UOT}(\mu_0,\mu_1) = \inf_{\gamma \in \mathcal{M}(X\times X)} \Rc(\mu_0,\mu_1 \mid \gamma),
    \end{equation*}
    where, recalling the Lebesgue decomposition in \eqref{L-decomp}, the reverse UOT functional $\Rc$ is given by 
    \begin{equation}\label{R}
    \Rc(\mu_0, \mu_1 \mid \gamma) := \int_{X \times X} \left( \sum_{i=0,1} R_i(\rho_i(x_i)) + c(x_0,x_1)\right)d\gamma(x_0,x_1) + R'_{\infty}\sum_{i=0,1}\left(\mu_i - \rho_i\gamma_i\right)(X),
    \end{equation}
    noting that $R_{\infty}' = F(0)$.
\end{theorem}
\subsubsection{Dual formulation}\label{sec:UOT-D}
We begin with a brief reminder on Convex Analysis, adapted from \cite[Appendix~A]{CPSV18}. Let $E$ and $E^*$ be topologically paired vector spaces, with the duality pairing between them denoted by $\langle \cdot, \cdot \rangle \, \colon\, E \times E^* \to \R$. The Legendre dual of a function $f\,\colon\, E \to \overline\R$ is defined, for each $x^* \in E^*$ by 
\begin{equation}\label{Legendre-dual}
f^*(x^*) := \sup_{x \in E}\, \langle x^*, x\rangle - f(x).
\end{equation}
\begin{theorem}[Fenchel-Rockafeller\cite{R67}]\label{thm:FRT}
Let $(E,E^*)$ and $(F,F^*)$ be two couples of topologically paired spaces and $A\,\colon\, E \to F$ be a continuous linear operator with $A^*\,\colon\, F^* \to E^*$ its adjoint. Suppose further that $f$ and $g$ are lower semicontinous and proper convex functions defined on $E$ and $F$ respectively. If there exists $x \in E$ such that $f(x) < \infty$ and $g$ is continuous at $Ax$, then 
\[
\sup_{x \in E} -f(-x) - g(Ax) = \min_{y^* \in F^*} f^*(A^*y^*) + g^*(y^*)
\]
and the minimum is attained. 
\end{theorem}
To derive the dual formulation, we apply this theorem to our setup by setting
\[
y^* \equiv \gamma,\quad A^*y^* \equiv (\gamma_0,\gamma_1),\quad f^*(A^*y^*) \equiv \F(\gamma_0 \mid \mu_0) + \F(\gamma_1 \mid \mu_1), \quad g^*(y^*) \equiv (c,\gamma),
\]
which in turn implies that
\[
    x \equiv (\phi_0,\phi_1), \quad Ax \equiv \phi_0 \oplus \phi_1, -f(-x) = \sum_i -\mathcal{F}^*(-\phi_i \mid \mu _i). 
\]
In particular, it follows from \cite[Proposition~A.3]{CPSV18} that the Legendre dual of the entropy functional $\F(\cdot \mid \mu_i)$ is given by 
\[
\F^*(\phi_i \mid \mu_i)  = \int_X F^*(\phi_i(x))d\mu_i(x) + \int_X \mathbbm{1}_{\leq F'_{\infty}}(\phi_i(x))d\mu_i(x).
\]
where, for a convex scalar function ${f\,\colon\, [0,+\infty) \to [0,+\infty]}$, its Legendre dual ${f^*\,\colon\,\R \to (\infty,+\infty]}$ can be seen from \eqref{Legendre-dual} to be given by
\begin{equation}\label{Legendre-dual-scalar}
f^*(\phi) = \sup_{s > 0}\left(s \phi - f(s)\right).
\end{equation}
Likewise, it follows from the scalar case, with $c \in \R$,  
\[
\eta \,\colon\, \R \to \R, \quad  \eta(s) := cs,  \quad \eta^*(\phi) := \sup_{s > 0}\left(s \phi - \eta(s)\right) \begin{cases} +\infty, \quad &\text{if } \phi > c \\ 0,\quad &\text{if } \phi \leq c \end{cases}
\]
that 
\[
    -g(Ax) = 0 \iff \phi_0 \oplus \phi_1 \leq c\quad \text{ and } -g(Ax) = -\infty\,\text{ otherwise.}
\]

To formalise the dual formulation, we thus define 
\begin{equation}
    \Phi := \{\phi = (\phi_0,\phi_1) \in C_b(X) \times C_b(X)\;\mid\; -F^*(-\phi_i) \in C_b(X),\;\phi_0 \oplus \phi_1 \leq c\}.
\end{equation}
and, owing to the change of variables
\begin{equation}\label{phi-psi-rel}
\phi_i:= R^*(\psi_i) \iff \psi_i = - F^*(-\phi_i),
\end{equation}
we further define 
\begin{equation}
    \Psi := \{\psi = (\psi_0,\psi_1) \in C_b(X) \times C_b(X)\;\mid\; R^*(\psi_i) \in C_b(X),\; R^*(\psi_0) \oplus R^*(\psi_1) \leq c\}.
\end{equation}
\begin{theorem}[\protect{\cite[Proposition 4.3, Theorem 4.11]{LMS18}}]
    The UOT problem \eqref{UOT-P} can be equivalently formulated as 
    \begin{equation}
        {\rm UOT}(\mu_0,\mu_1) = \sup_{\psi \in \Psi} \D_{\psi}(\psi \mid \mu_0, \mu_1) = \sup_{\phi \in \Phi} \D_{\phi}(\phi \mid \mu_0, \mu_1),
    \end{equation}
    where the two dual UOT functionals $\D_\psi$ and $\D_\phi$ are given by 
    \begin{equation}\label{D}
    \D_{\psi}(\psi \mid \mu_0, \mu_1) = \sum_{i=0,1} \mu_i\big(\psi_i\big), \quad \D_{\phi}(\phi \mid \mu_0, \mu_1) = \sum_{i=0,1} \mu_i\big(-F^*(-\phi_i)\big). 
    \end{equation}
\end{theorem}

An explicit example to have in mind here is the KL case when $F$ is given by \eqref{KL} and $R$ by \eqref{R-KL}, with their Legendre duals given by 
\begin{equation}\label{F-R-dual}
F^*(\phi) = \exp(\phi) - 1, \quad R^*(\psi) = -\log(1-\psi),\end{equation}
for which the change of variables formulae in \eqref{phi-psi-rel} can be readily verified to hold. 
\subsubsection{Homogenous formulation}\label{sec:UOT-H}
The last and arguably most fascinating equivalent formulation concerns introducing the induced marginal perspective cost function 
\[
H\,\colon\, (X \times \R_+) \times (X \times \R_+) \to [0,+\infty],
\]
defined by
\begin{equation}
H(x_0,s_0,x_1,s_1):= \inf_{t >0} \Bigg\{t\left(R\left(\frac{s_0}{t}\right) + R\left(\frac{s_1}{t}\right) + c(x_0,x_1)\right)\Bigg\}.
\end{equation}
It can be shown \cite[Lemma~5.3]{LMS18} that $H$ admits a dual representation akin to the discussion in Section~\ref{sec:UOT-D}, namely, for $c(x_0,x_1)$ fixed,
\begin{align*}
H(x_0,s_0,x_1,s_1) &= \sup \big\{ s_0 \psi_0 + s_1\psi_1\, \mid\, R^*(\psi_0) + R^*(\psi_1) \leq c(x_0,x_1) \big\}\\
&= \sup \big\{ -s_0 F^*(-\phi_0) - s_1F^*(-\phi_1)\, \mid\, \phi_0 + \phi_1 \leq c(x_0,x_1) \big\}.
\end{align*}
It is through this dual characterisation that one can reasonably easily establish the following.
\begin{theorem}[\protect{\cite[Theorem~5.5]{LMS18}}]
    The UOT problem \eqref{UOT-P} can be equivalently stated as 
    \begin{equation}
        {\rm UOT}(\mu_0,\mu_1) = \inf_{\gamma \in \M(X\times X)} \H(\mu_0,\mu_1 \mid \gamma)
    \end{equation}
    where, recalling the Lebesgue decomposition in \eqref{L-decomp}, the homogeneous UOT functional $\H$ is given by
    \begin{equation}\label{H}
    \H(\mu_0,\mu_1 \mid \gamma) := \int_{X \times X} H(x_0,\rho_0(x_0), x_1, \rho_1(x_1) d\gamma(x_0,x_1) + F(0) \sum_{i=0,1}\left(\mu_i - \rho_i\gamma_i\right)(X).
    \end{equation}
\end{theorem}
We note that in the KL case \eqref{KL}, \eqref{R-KL}, the induced marginal perspective cost function is given by
\begin{equation}\label{H-KL}
H(x_0,s_0,x_1,s_1) = s_0 + s_1 - 2\sqrt{s_0s_1}\exp\left(\frac{-c(x_0,x_1)}{2}\right).
\end{equation}
\subsection{Extended space homogeneous formulation}\label{sec:UOT-ExtH}
A pair $(\gamma_i,\rho_i)$ from the Lebesgue decomposition \eqref{L-decomp} gives rise to a measure $\beta \in \M(Y)$ defined on the extended space
\begin{equation}\label{def:Y}
Y:= X \times \R_+,\quad y = (x,s) \in Y
\end{equation}
via
\[
\beta := (x,\rho_i(x))_{\#}\gamma_i.
\]
It is just one example among a family of nonnegative measures on $Y$ lying in the space
\begin{equation}\label{def:Mp}
\M_p(Y) := \Big\{ \beta \in \M(Y) \,\Bigm| \, \int_{Y} s^p d\beta < \infty\Big\}
\end{equation}
satisfying the $p$th homogeneous marginal constraint (with $p=1$ above) given by ${\h^p\beta = \rho_i\gamma_i}$, where
\[
\h^p \,\colon \M_p(Y) \to \M(X), \quad \h^p\beta := \pi^x_{\#}(s^p \beta) \in \M(X).
\]

Analogously, any measure $\gamma \in \M(X\times X)$ and the resulting pairs $(\gamma_0,\rho_0)$ and $(\gamma_1,\rho_1)$ in the Lebesgue decomposition \eqref{L-decomp} give rise to a measure
\[
\alpha \in \M(Y \times Y), \quad \alpha := (x_0,\rho_0(x_0),x_1,\rho_1(x_1)_{\#}\gamma
\]
and, using the notation $(y_0,y_1) \in Y \times Y,\; y_i = (x_i,s_i)$, we can define
\[
\h_i^p\, \colon\,\M_p(Y\times Y) \to \M(X), \quad \h_i^p\alpha := \pi^{x_i}_{\#}(s_i^p \alpha) \in \M(X),
\]
where 
\begin{equation}\label{def:Mp_prod}
\M_p(Y\times Y) := \Big\{ \alpha \in \M(Y\times Y) \,\Bigm| \, \int_{Y\times Y} \left(s_0^p + s_1^p\right) d\alpha < \infty\Big\}.
\end{equation}
Since by Lebesgue decomposition we have $\mu_i = \rho_i \gamma_i + \mu_i^{\perp}$ and the singular part $\mu_i^{\perp}$ is nonnegative, then considering a family of measures satisfying $\h_i^p \alpha = \rho_i\gamma_i$ is equivalent to the restriction that $\h_i^p\alpha_i \leq \mu_i$ and hence we are interested in sets
\begin{align*}
S^p_{\leq}(\mu_0,\mu_1) &:= \{ \alpha \in \M(Y \times Y) \, \mid\, \h_i^p \alpha \leq \mu_i \},\\
S^p_{=}(\mu_0,\mu_1) &:= \{ \alpha \in \M(Y \times Y) \, \mid\, \h_i^p \alpha = \mu_i \},
\end{align*}
where, clearly, 
\[
S^p_{=}(\mu_0,\mu_1) \subset S^p_{\leq}(\mu_0,\mu_1) \subset \M_p(Y\times Y).
\]
\begin{theorem}[\protect{\cite[Theorem~5.8]{LMS18}}]\label{thm:UOT-H-ext}
    The UOT problem \eqref{UOT-P} can be equivalently formulated as 
    \begin{align*}
        {\rm UOT}(\mu_0,\mu_1) &= \inf_{\alpha \in S^p_{\leq}(\mu_0,\mu_1)}\Big\{ (H_p,\alpha) + F(0) \sum_{i=0,1}\left(\mu_i - \h_i^p\alpha\right)(X) \Big\},\\
        &=\inf_{\alpha \in S^p_{=}(\mu_0,\mu_1)} (H_p,\alpha)
    \end{align*}
    where
    \[
    (H_p,\alpha) := \int_{Y\times Y}H_p(y_0,y_1) d\alpha(y_0,y_1), \quad H_p(y_0,y_1) = H(x_0,s_0^p,x_1,s_1^p).
    \]
\end{theorem}
\subsubsection{Rescaling invariance}\label{sec:UOT-ExtH-rescale}
The following crucial rescaling invariance  result holds. 
\begin{prop}[\protect{\cite[Section~5.2]{LMS18}}]\label{prop:rescale}    Fix $p >0$, $\mu_0,\mu_1 \in \M(X)$ and set
    \[
    s_* := (\mu_0(X)+\mu_1(X))^{\frac 1 p}.
    \]
    Consider functions $\theta\,\colon\,Y \times Y \to [0,\infty)$ and ${\rm prd}_{\theta}\,\colon\, Y \times Y \to Y \times Y$ given by
    \begin{equation}\label{prod-theta}
    \theta(y_0,y_1) = \frac{1}{s_*}(s_0^p + s_1^p)^{\frac 1 p}, \quad {\rm prd}_{\theta}(x_0,s_0,x_1,s_1) := \left(x_0,\frac{s_0}{\theta},x_1,\frac{s_1}{\theta}\right). 
    \end{equation}
    Suppose further that $\alpha \in S^p_{\leq}(\mu_0,\mu_1)$ and  define     
    \[
    \tilde \alpha := \alpha\,\mres\, (Y \times Y) \setminus \{ y \in Y \times Y \,\mid\, s_0 = s_1 = 0\} ,\quad \hat \alpha := ({\rm prd}_\theta)_{\#}(\theta^p \tilde \alpha).
    \]
    Then
    \begin{itemize}
        \item $(H_p,\alpha) = (H_p,\tilde \alpha) = (H_p,\hat \alpha)$,
        \item $\hat \alpha \in S^p_{\leq}(\mu_0,\mu_1)$ (and if $\alpha \in S^p_{=}(\mu_0,\mu_1)$ then $\hat \alpha \in S^p_{=}(\mu_0,\mu_1)$),
        \item $\hat \alpha \in \Pc(Y \times Y)$,
        \item $\hat \alpha ((Y\times Y)[s_*]') = 0$, where
        \begin{alignat}{2}
        Y[s_*] &:= \{ (x,s) \in Y \mid s \leq s_*\}, \quad &&Y[s_*]' = Y \setminus Y[s_*],\\
        (Y\times Y)[s_*] &:= Y[s_*] \times Y[s_*], \quad &&(Y\times Y)[s_*]' = Y\times Y \setminus (Y\times Y)[s_*].
        \end{alignat} 
    \end{itemize}
\end{prop}
Proposition~\ref{prop:rescale} ensures the following result. 
\begin{theorem}\label{thm:UOT-rescaled}
The UOT problem \eqref{UOT-P} can be equivalently stated as 
\begin{equation}\label{UOT-rescaled}
{\rm UOT}(\mu_0,\mu_1) =\inf_{\alpha \in \overline{S}^p_{=}(\mu_0,\mu_1)} (H_p,\alpha),
\end{equation}
where 
\[
\overline{S}^p_{=}(\mu_0,\mu_1) := \{ \alpha \in \Pc(Y \times Y) \, \mid\, \h_i^p \alpha = \mu_i,\; \alpha((Y\times Y)[s_*]') = 0\}.
\]
\end{theorem}
\begin{cor}[\protect{\cite[Corollary~7.7]{LMS18}}]\label{thm:UOT_as_OT}
A further reformulation of the UOT problem \eqref{UOT-P} is given by
\[
{\rm UOT}(\mu_0,\mu_1) = \inf_{(\beta_0,\beta_1)} \big\{{\rm OT}(\beta_0,\beta_1)\, \mid\, \beta_i \in \overline{S}^p_{=}(\mu_i)\big\},
\]
where
\[
{\rm OT}(\beta_0,\beta_1) := \inf_{\alpha \in \Gamma(\beta_0,\beta_1)} (H_p, \alpha),
\]
\[
\Gamma(\beta_0,\beta_1) := \{\alpha \in \Pc(Y\times Y),\; \pi^i_{\#}\alpha = \beta_i\}
\]
is the set of couplings, and
\[
\overline{S}^p_{=}(\mu_0,\mu_1) := \{ \beta \in \Pc(Y) \, \mid\, \h^p \beta = \mu_i,\; \beta((Y[s_*]') = 0\}.
\]
\end{cor}
\begin{prop}[\protect{\cite[Remark~5.10,Theorem~5.8]{LMS18}}]\label{prop:alpha_bar}
There exists $\overline{\alpha} \in \overline{S}^p_{=}(\mu_0,\mu_1)$ minimising the right-hand side of \eqref{UOT-rescaled}.

Furthermore, if $\gamma \in \M(X \times X)$ is a minimiser of the primal formulation of the UOT problem, defined in \eqref{UOT-P}, existence of which is guaranteed by Theorem~\ref{thm:UOT-P-minimiser}, then ${\overline{\alpha} \in \M(Y \times Y)}$ can be expressed as
\[
\overline{\alpha} = ({\rm prd}_\theta)_{\#}(\theta^p \tilde \alpha),
\]
with ${\rm prd}_\theta$ defined in \eqref{prod-theta} and ${\alpha \in S^p_{=}(\mu_0,\mu_1)}$ is given by 
\[
\alpha = (x_0,\rho^{1/p}_0(x_0),x_1,\rho^{1/p}_1(x_1))_{\#}\gamma,
\]
where the densities $\rho_i$ are obtained from the Lebesgue decomposition \eqref{L-decomp}.
\end{prop}

This concludes the brief exposition of all the relevant results for the unregularised unbalanced optimal transport problems. 

\subsubsection{Example: Lifting of the balanced optimal transport}\label{sec:OT}

We finish this section by discussing the special case of the balanced optimal transport problem, emphasising that even there the idea of lifting to higher dimensional spaces is not redundant. In particular, this example will set the scene for the idea of subsequent liftings, to be explored in Section~\ref{sec:eUOT-unify}.

The unregularised unbalanced optimal transport problem introduced Section~\ref{sec:UOT-P} reduces to balanced optimal transport in the case when the entropy function $F$ from \eqref{F-func} is set to be
\[
\overline F(s) := \begin{cases}
    0,\quad &\text{when } s=1, \\ +\infty,\quad &\text{otherwise.}
\end{cases}
\]
To be precise, we consider, for $\mu_0,\mu_1 \in \M(X \times X)$,
\begin{equation*}
{\rm OT}(\mu_0,\mu_1) := \inf_{\gamma \in \mathcal{M}(X\times X)} \overline\E(\gamma \mid \mu_0,\mu_1),
\end{equation*}
where  
\begin{equation*}
\overline\E(\gamma \mid \mu_0, \mu_1) := \sum_{i=0,1} \overline{\mathcal{F}}(\gamma_i \mid \mu_i) + (c,\gamma), \quad \overline{\mathcal{F}}(\gamma_i \mid \mu_i) = \int_X \overline F(\sigma_i)d\mu_i + \overline F'_{\infty}\gamma_i^{\perp}(X).
\end{equation*}
In particular we note that if $\mu_0(X) \neq \mu_1(X)$, then ${\rm OT}(\mu_0,\mu_1) = \infty$. Nonetheless, all the results presented so far in Section~\ref{sec:UOT} apply and hence we observe that 
\begin{equation}\label{eqn:OT_H}
    {\rm OT}(\mu_0,\mu_1) = \inf_{\gamma \in \M(X \times X)} \int_{X \times X}H(x_0,\rho_0(x_0),x_1,\rho_1(x_1))d\gamma
\end{equation}
where, due to entropy functions $\overline F$ for marginals, we have 
\[
H(x_0,s_0,x_1,s_1) = \begin{cases}
    ac(x_0,x_1),\quad&\text{when } s_0 = s_1 = a,\\
    +\infty,\quad&\text{otherwise.}
\end{cases}
\]
In particular, we note that a necessary condition on the marginals of $\gamma$ for the integral in \eqref{eqn:OT_H} to be finite is that $a \gamma_i = \mu_i$ for some constant $a > 0$ and we can split the infimisation into two steps and observe that
\[
    {\rm OT}(\mu_0,\mu_1) = \inf_{\gamma \in \Gamma(\mu_0,\mu_1)} \Big\{\inf_{a> 0} \int_{X \times X}\frac{1}{a}c(x_0,x_1)d (a\gamma) \Big\} = \inf_{\gamma \in \Gamma(\mu_0,\mu_1)} \int_{X \times X} c(x_0,x_1)d\gamma,
\]
thus just reducing to the usual formulation. 

Once we lift to the extended space, such a full reduction does not hold. We have
\begin{equation}\label{eqn:OT-H-ext}
{\rm OT}(\mu_0,\mu_1) = \inf_{\alpha \in S^p_=(\mu_0,\mu_1)} \int_{Y \times Y}H(x_0,s_0,x_1,s_1)d\alpha.
\end{equation}
Clearly a necessary condition on the support of $\alpha$ for the integral to be finite is that 
\[
\alpha((Y\times Y)\setminus \Omega) = 0, \quad \text{where } \Omega = \{(y_0,y_1) \in Y \times Y \mid s_0 = s_1\}.
\]
In other words, $\alpha$ has to be of the form
\[
\alpha = (x_0,s,x_1,s)_{\#}\beta, \quad \beta \in \M(X^2 \times \R_+),\; (x_0,x_1,s) \in X^2\times \R_+.
\]
Since $\alpha \in S^p_{=}(\mu_0,\mu_1)$, the corresponding $\beta$ has to belong to 
\[
\widetilde S^p_{=}(\mu_0,\mu_1) := \left\{ \beta \in \M(X^2\times \R_+)\, \Big|\, \pi^{x_i}_{\#}(s^p\beta) = \mu_i \right\},
\]
where, for any measurable function $f\,\colon\,X \to \R$,
\[
\int_X f(x) d\pi^{x_i}_{\#}(s^p\beta) = \int_{X \times X \times \R_+} s^p f(x_i)d\beta.
\]
It readily follows that 
\[
{\rm OT}(\mu_0,\mu_1) = \inf_{\beta \in \widetilde S^p_=(\mu_0,\mu_1)}\int_{X \times X \times \R_+} s\, c(x_0,x_1)d\beta.
\]

We thus see that in the balanced case, the lifting strategy takes us from a minimisation problem posed on $\M(X \times X)$ to a minimisation problem posed on \linebreak ${\M(X \times X \times \R_+)}$.

\section{Entropic regularisation of unbalanced optimal transportation}\label{sec:eUOT}
In this main section of our paper, we study and obtain theoretical results about the different ways in which the unbalanced optimal transport problem can be entropy-regularised.
\subsection{Regularisation on the original space}\label{sec:eUOT_X}
The standard and seemingly most natural way to introduce entropic regularisation of the UOT problem \eqref{UOT-P} is to consider, for some appropriate reference measure $\nu_X \in \M(X \times X)$ (the subscript is kept to distinguish $\nu_X$ from $\nu_Y$ to be introduced in Section~\ref{sec:eUOT_Y}) and some small $\e >0$,
\begin{equation}\label{eUOT-P}
{\rm UOT}_{X,\e}(\mu_0,\mu_1) := \inf_{\gamma \in \mathcal{M}(X\times X)}\E_\e(\gamma \mid \mu_0,\mu_1, \nu_X),
\end{equation}
where, recalling the definition of the primal UOT functional $\E$ in \eqref{P}, its regularised counterpart $\E_\e$ is defined by
\begin{equation}\label{eP}
\E_\e(\gamma \mid \mu_0, \mu_1, \nu_X) := \E(\gamma \mid \mu_0, \mu_1)+\e \F(\gamma \mid \nu_X).
\end{equation}
Here, similarly to \eqref{F-func}, we have
\begin{equation*}
    \F(\gamma\mid \nu_X) := \int_{X\times X} F(\varsigma)d\nu_X + F_{\infty}' \gamma^{\perp}(X \times X),
\end{equation*}
with $F$, both here and in the definition of $\F(\gamma_i \mid \mu_i)$, taken to be the KL divergence \eqref{KL} and, as already mentioned, $\nu_X \in \M(X\times X)$ is some a-priori fixed reference measure, with the density $\varsigma$, similarly to \eqref{L-decomp}, obtained via Lebesgue decomposition
\begin{equation}\label{L-decomp-full}
\gamma = \varsigma \nu_X + \gamma^{\perp}, \quad \nu_X = \varrho \gamma + \nu_X^{\perp}.
\end{equation}

There is by now a substantial body of literature devoted to studying ${\rm UOT}_{X,\e}$ and closely related problems \cite{CPSV18,SFVTP19,SPV22,L22}, on which we now make several remarks.

\begin{remark}
For future reference, we explicitly note that in \cite{CPSV18} the authors require the reference measure $\nu_X$ to be a probability measure and consider
\[
\widetilde{\rm UOT}_{X,\e}(\mu_0,\mu_1):= \inf_{\gamma \in \M(X \times X)} \left\{ \sum_i \F(\gamma_i \mid \mu_i) + (c,\gamma) + \e\tilde\F(\gamma \mid \nu_X)\right\}
\]
where
\[
\tilde \F(\gamma \mid \nu_X) := \int_{X \times X} \tilde F(\sigma) d\nu_X + \tilde{F}'_{\infty}\gamma^{\perp}(X\times X), \quad \tilde F(s) := s\log s -s,
\]
and, owing to the introduction of a new reference measure $\overline \nu_X:= \exp\left(-
\frac{c}{\e}\right)\nu_X$, the other problem considered therein is given by
\[
\overline{\rm UOT}_{X,\e}(\mu_0,\mu_1):= \inf_{\gamma \in \M(X \times X)} \left \{ \sum_i \F(\gamma_i \mid \mu_i) + \e\F(\gamma \mid \overline \nu_X) \right\}.
\]
It can be readily established that 
\begin{align*}
{\rm UOT}_{X, \e}(\mu_0,\mu_1) &= \widetilde{\rm UOT}_{X,\e}(\mu_0,\mu_1) + \e \nu_X(X \times X) \\
&= \overline{\rm UOT}_{X,\e}(\mu_0,\mu_1) - \e \int_{X\times X} \exp\left(\frac{-c}{\e}\right) d\nu_X +  \e \nu_X(X \times X). 
\end{align*}
In a more recent effort \cite{SFVTP19} the authors consider ${\rm UOT}_{X,\e}(\mu_0,\mu_1)$ with reference measure $\nu_X$ chosen to be the product measure $\nu_X = \mu_0 \otimes \mu_1$. This has an obvious limitation of introducing unnecessary mass discrepancy since $\nu_X(X \times X) = \mu_0(X) \mu_1(X)$ and 
\[
\pi^i_{\#}\nu_X = \mu_j(X) \mu_i \neq \mu_i,
\]
unless $\mu_j(X) = 1$.
\end{remark}
\begin{remark}
In the recent interesting effort aimed at deriving a regularisation of unbalanced transport problems that preserves homogeneity \cite{L22}, the author considers 
\[
{\rm UOT}_{X,\e}^{(H)}(\mu_0,\mu_1):= \inf_{\gamma \in \M(X \times X)} \left\{ \sum_i \F(\gamma_i \mid \mu_i) + (c,\gamma) + \e \G(\gamma \mid \mu_0, \mu_1)\right\},
\]
where
\[
\G(\gamma \mid \mu_0, \mu_1) := \frac12\left(\F(\gamma \mid \hat \mu_0 \otimes \mu_1) + \F(\gamma \mid \mu_0 \otimes \hat \mu_1)\right), \quad \hat \mu_i := \frac{\mu_i}{\mu_i(X)} \in \Pc(X).
\]
It follows from careful rewriting that, using the Lebesgue decomposition in \eqref{L-decomp-full} with $\nu_X = \mu_0 \otimes \mu_1$, the functional $\G$ can be rewritten as
\[
\G(\gamma \mid \mu_0, \mu_1) = \int_{X\times X}G(\sigma) d (\mu_0 \otimes \mu_1) + G_{\infty}' \gamma^{\perp}(X\times X), 
\]
where
\[
G(s) := s \log s + (c_0 - 1)s + \frac1{c_1}, \quad c_0 = \frac12\log(\mu_0(X)\mu_1(X)),\,c_1 = 2 \frac{\mu_0(X)\mu_1(X)}{\mu_0(X)+\mu_1(X)}.
\]
The form of $c_0$ and $c_1$ appears somewhat reminiscent of the rescalling invariance to be discussed in Section~\ref{sec:eUOT-H-rescale}
\end{remark}

To the best of our knowledge there have so far been no results concerning the existence of minimisers of the problem \eqref{eUOT-P}, which we will now establish by arguing as in \cite[Theorem 3.3, Theorem 6.2]{LMS18}.
\begin{theorem}[Existence of $X$-space-regularised minimisers]
Under natural assumptions on the entropy functional $\F$, the cost function $c$, and the reference measure $\nu_X$, there exists at least one $\gamma_\e \in \M(X\times X)$ minimizing the right-hand side of \eqref{UOT-P}.
\end{theorem}
\begin{proof}
{\cb \footnotesize [Statement of the theorem to be made precise and the proof to be added in the next version of the manuscript.]}
\end{proof}
Mimicking the approach outlined in Section~\ref{sec:UOT}, we will now present equivalent formulations of ${\rm UOT}_{X,\e}$, including the new homogeneous and extended-space formulations, which leverage the interplay between the (obvious) reverse formulation and the (already known \cite{CPSV18}) dual formulation. To the best of our knowledge, this is the first attempt at doing so.

\subsubsection{Reverse formulation}\label{sec:eUOT-R}
Recalling the definition of the reverse UOT functional $\Rc$ in \eqref{R}, we introduce its regularised counterpart as
\begin{align}\label{eR}
\Rc_\e(\mu_0, \mu_1, \nu_X \mid \gamma) &:= \Rc(\mu_0, \mu_1 \mid \gamma)+ \e\int_{X \times X} R(\varrho(x_0,x_1)) d\gamma(x_0,x_1) \\ &+ \e R'_{\infty}\left(\nu_X - \varrho\gamma\right)(X\times X),\nonumber
\end{align}
The following proposition is an immediate extension of the  discussion presented in Section~\ref{sec:UOT-R}
\begin{prop}\label{prop:eUOT-R}
The regularised unbalanced optimal transport problem ${\rm UOT}_{X,\e}$, defined in \eqref{eUOT-P}, can be equivalently stated as
\begin{equation*}
    {\rm UOT}_{X,\e}(\mu_0,\mu_1) = \inf_{\gamma \in \mathcal{M}(X\times X)} \Rc_\e(\mu_0,\mu_1, \nu_X \mid \gamma).
\end{equation*}
\end{prop}
\subsubsection{Dual formulation}\label{sec:eUOT-D}
Recalling the discussion in Section~\ref{sec:UOT-D}, an application of Theorem~\ref{thm:FRT}, as already discussed in \cite{CPSV18}, yields the following result. 
\begin{prop}\label{prop:eUOT-D}
The regularised unbalanced optimal transport problem ${\rm UOT}_{X,\e}$, defined in \eqref{eUOT-P}, can equivalently stated as
\begin{equation*}
    {\rm UOT}_{X,\e}(\mu_0,\mu_1) = \sup_{\phi \in \tilde\Phi} \D_{\phi,\e}(\phi \mid \mu_0,\mu_1,\nu_X),
\end{equation*}
where
\begin{equation}\label{eD}
\D_{\phi,\e}(\phi \mid \mu_0,\mu_1,\nu_X) := \sum_{i} \mu_i(-F^*(-\phi_i)) + \e\nu_X\left(-F^*\left(\frac{\phi_0 \oplus \phi_1 - c}{\e}\right)\right)
\end{equation}
and
\begin{equation}\label{tildePhi}
    \tilde\Phi := \{\phi = (\phi_0,\phi_1) \in C_b(X) \times C_b(X)\;\mid\; -F^*(-\phi_i) \in C_b(X)\}
\end{equation}
\end{prop}

\subsubsection{Homogenous formulation}\label{sec:eUOT-H}
Mimicking the approach outlined in Section~\ref{sec:UOT-H}, we now introduce the \emph{regularised induced marginal perspective cost function}
\[
H_\e \,\colon\, (X \times \R_+) \times (X \times \R_+) \times \R_+ \to [0,+\infty],
\]
as
\begin{equation}\label{He-def}
H_\e(x_0,s_0,x_1,s_1,S):= \inf_{t >0} \Bigg\{t\Big(R\left(\frac{s_0}{t}\right) + R\left(\frac{s_1}{t}\right) + \e R\left(\frac{S}{t}\right) + c(x_0,x_1)\Big)\Bigg\}.
\end{equation}
We prove the following (c.f. \eqref{H-KL} for the unregularised case). 
\begin{lemma}\label{lem:eH-explicit}
If $R$ is given by \eqref{R-KL}, then the regularised induced marginal perspective cost function introduced in \eqref{He-def} admits a representation as
\[
H_\e(x_0,s_0,x_1,s_1,S) = s_0 + s_1 + \e S - (2+\e)\left((s_0s_1)^{\frac{1}{2+\e}} S^{\frac{\e}{2+\e}}\exp\left(\frac{-c(x_0,x_1)}{2+\e}\right)\right).
\]
\end{lemma}
\begin{proof}
It follows from a direct calculation by differentiating the right-hand side of \eqref{He-def}.
\end{proof}

Another technical ingredient we prove is the following. 
\begin{lemma}\label{lem:eH-dual}
The regularised induced marginal perspective cost function introduced in \eqref{He-def} admits a dual representation which, for $c(x_0,x_1)$ fixed, is given by
\begin{align*}
H_{\e}(x_0,s_0,x_1,s_1,S)
= \sup_{(\phi_0,\phi_1) \in \R^2_+} \Bigg\{ &-s_0 F^*(-\phi_0) - s_1F^*(-\phi_1)\\ &- \e S F^*\Big(\frac{\phi_0+\phi_1 - c(x_0,x_1)}{\e}\Big)\Bigg\}.
\end{align*}
\end{lemma}
\begin{proof}
For brevity we suppress the dependence of $H_\e$ and $c$ on its variables. By the definition of the reverse entropy function $R$ in terms of the entropy $F$ given in \eqref{R-ent}, we have 
\[
H_\e = \inf_{t>0} \Bigg\{s_0F\left(\frac{t}{s_0}\right) + s_1F\left(\frac{t}{s_1}\right) + \e SF\left(\frac{t}{S}\right) + tc\Bigg\}.
\]
Set $\xi(t) := \e SF\left(\frac{t}{S}\right) + tc$. The Legendre dual of $\xi$, as in \eqref{Legendre-dual-scalar}, is given by
\[
\xi^*(\phi) := \sup_{t > 0}\Bigg( t\phi - \e SF\left(\frac{t}{S}\right) - tc\Bigg) =: \sup_{t > 0} \tilde\xi(t).
\]
Recalling that $F(s) = s\log s -s + 1$, it is not hard to see that 
\[
\tilde\xi'(\overline t) = 0 \implies \overline t = S\exp\left(\frac{\phi - c}{\e}\right),
\]
which, by plugging this formula for $\overline t$, lets us conclude that 
\[
\xi^*(\phi) = -\e S\left(-\exp\left(\frac{\phi-c}{\e}\right)+1\right).
\]
In preparation to apply Theorem~\ref{thm:FRT}, we identify
\[
y^* \equiv t \in \R_+,\; A^*y^* \equiv \left(\frac{t}{s_0},\frac{t}{s_1}\right) \in \R^2_+,\; f^*(A^*y^*) \equiv s_0 F\left(\frac{t}{s_0}\right) + s_1 F\left(\frac{t}{s_1}\right), \; g^*(y^*) \equiv \xi(t),
\]
which in turn implies that
\[
x \equiv (\phi_0,\phi_1) \in\R_+^2, \quad Ax \equiv \frac{\phi_0}{s_0} + \frac{\phi_1}{s_1}.
\]
We further note that, since $f^*\,\colon\,\R^2_+ \to \R$ is given by $f^*(x^*) = s_0F(x_0^*) + s_1 F(x_1^*)$, we have
\begin{align*}
f(x) &= \sup_{x^*\in\R^2_+} x\cdot x^* - f^*(x^*) = \left(\sup_{x_0^* > 0} x_0^*x_0 - s_0F(x_0^*)\right) + \left(\sup_{x_1^* > 0} x_1^*x_1 - s_1F(x_1^*)\right)\\
&= (s_0 F)^*(x_0) + (s_1 F)^*(x_1) \equiv s_0F^*\left(\frac{\phi_0}{s_0}\right) + s_1F^*\left(\frac{\phi_1}{s_1}\right)
\end{align*}
where the last equality follows from the fact that, in general,
\[
(aF)^*(\phi) = a F^*\left(\frac{\phi}{a}\right),
\]
which implies that 
\[
-f(-x) \equiv -s_0 F^*\left(-\frac{\phi_0}{s_0}\right) - s_1 F^*\left(-\frac{\phi_1}{s_1}\right).
\]
A direct application of Theorem~\ref{thm:FRT} thus implies
\begin{align*}
H_{\e}(x_0,s_0,x_1,s_1,S)
= \sup_{(\phi_0,\phi_1)\in \R^2_+} \Bigg\{ &-s_0 F^*\left(-\frac{\phi_0}{s_0}\right) - s_1F^*\left(-\frac{\phi_1}{s_1}\right)\\ &- \e S F^*\Big(\frac{\frac{\phi_0}{s_0}+\frac{\phi_1}{s_1} - c(x_0,x_1)}{\e}\Big)\Bigg\}.
\end{align*}
The result of the lemma is thus established by recognising that we can safely substitute $\frac{\phi_i}{s_i} \mapsto \phi_i$.
\end{proof}
We are now in a position to introduce the regularised homogeneous UOT functional $\H_\e$ as
\begin{align}\label{eH}
\H_\e(\mu_0,\mu_1, \nu_X \mid \gamma) := &\int_{X \times X} H_\e(x_0,\rho_0(x_0), x_1, \rho_1(x_1),\varrho(x_0,x_1)) d\gamma(x_0,x_1)\\ &+ F(0)\left( \sum_{i=0,1}\left(\mu_i - \rho_i\gamma_i\right)(X) + \left(\nu_X - \varrho\gamma\right)(X \times X) \right)
\end{align}
and, owing to Lemma~\ref{lem:eH-dual}, we will now prove the main result of this section.
\begin{theorem}\label{thm:eUOT-X-H}
    The regularised unbalanced optimal transport problem ${\rm UOT}_{X,\e}$, defined in \eqref{eUOT-P}, can be equivalently stated as
    \begin{equation*}
        {\rm UOT}_{X,\e}(\mu_0,\mu_1) = \inf_{\gamma \in \M(X \times X)} \H_\e(\mu_0,\mu_1, \nu_X \mid \gamma).
    \end{equation*}
\end{theorem}
\begin{proof}
    Using the strategy employed in the proof of \cite[Theorem~5.5]{LMS18}, we will show that $\mu_i \in \M(X)$, $\nu_X \in \M(X\times X)$ and $\phi \in \tilde\Phi$ (defined in \eqref{tildePhi}),
    \begin{equation}\label{eUOT-X-H_ineq}
    \D_{\phi,\e}(\phi \mid \mu_0,\mu_1,\nu_X) \leq \H_\e(\mu_0,\mu_1, \nu_X \mid \gamma) \leq \Rc_\e(\mu_0,\mu_1, \nu_X \mid \gamma),
    \end{equation}
    which ensures the result of the theorem, as then 
    \[
    \underbrace{\sup_{\phi \in \tilde\Phi}\D_{\phi,\e}(\phi \mid \mu_0,\mu_1,\nu_X)}_{={\rm UOT}_{X,\e}(\mu_0,\mu_1)}\leq \inf_{\gamma \in \M(X \times X)} \H_\e(\mu_0,\mu_1, \nu_X \mid \gamma) \leq \underbrace{\inf_{\gamma \in \M(X\times X)}\Rc_\e(\mu_0,\mu_1, \nu_X \mid \gamma)}_{={\rm UOT}_{X,\e}(\mu_0,\mu_1)}.
    \]
    
    In fact, the second inequality in  \eqref{eUOT-X-H_ineq} is immediate, as we recover $\Rc_\e$ simply by setting $t = 1$ inside the infimum on the right-hand-sid \eqref{He-def}.

    To prove the first inequality in \eqref{eUOT-X-H_ineq}, we leverage the dual description of $H_\e$ obtained in Lemma~\ref{lem:eH-dual} and observe that 
    \begin{align*}
        \H_\e(\mu_0,\mu_1, \nu_X \mid \gamma) \geq &\int_{X \times X} -\rho_0(x_0) F^*(-\phi_0(x_0)) d\gamma(x_0,x_1) \\ + &\int_{X\times X} -\rho_1(x_1) F^*(-\phi_1(x_1))d\gamma(x_0,x_1) \\ + &\int_{X\times X}-\e \varrho(x_0,x_1) F^*\left(\frac{\phi_0(x_0) + \phi_1(x_1) - c(x_0,x_1)}{\e}\right)d\gamma(x_0,x_1)\\ + &F(0)\left( \sum_{i=0,1}\left(\mu_i - \rho_i\gamma_i\right)(X) + \left(\nu_X - \varrho\gamma\right)(X \times X) \right)\\
        = &\sum_i \mu_i(-F^*(-\phi_i)) + \e\nu_X\left(-F^*\left(\frac{\phi_0 \oplus \phi_1 - c}{\e}\right)\right)\\
        +&\sum_i(\rho_i\gamma_i-\mu_i)(-F^*(-\phi_i)) + \e(\varrho\gamma - \nu_x)\left(-F^*\left(\frac{\phi_0 \oplus \phi_1 - c}{\e}\right)\right)\\
        + &F(0)\left( \sum_{i=0,1}\left(\mu_i - \rho_i\gamma_i\right)(X) + \left(\nu_X - \varrho\gamma\right)(X \times X) \right)\\
        \geq &\sum_i \mu_i(-F^*(-\phi_i)) + \e\nu_X\left(-F^*\left(\frac{\phi_0 \oplus \phi_1 - c}{\e}\right)\right) \\ &= \D_{\phi,\e}(\phi \mid \mu_0,\mu_1,\nu_X),
    \end{align*}
where the last inequality follows from $\sup -F^* = -F(0)$ (a standard property of Legendre duality, as discussed e.g. in \cite[Section~2.3]{LMS18}.
\end{proof}
\subsubsection{Extended space homogeneous formulation}\label{sec:eUOT-ExtH}

A measure $\gamma \in \M(X\times X)$ and the triplet of densities $(\rho_0,\rho_1,\rho)$ from the Lebesgue decomposition for the marginals $\gamma_i$, in \eqref{L-decomp}, and $\gamma$ itself, in \eqref{L-decomp-full}, together give rise to a measure $\eta \in \M(Y\times Y \times \R_+)$ (recall the definition of $Y$ in \eqref{def:Y}) defined as
\[
\eta = (x_0,\rho_0(x_0),x_1,\rho_1(x_1), \rho(x_0,x_1))_{\#}\gamma 
\]
It is just one example among a family of nonnegative measures on $Y \times Y \times \R_+$ lying in the space
\begin{equation}\label{def:MpYYR}
\M_p(Y\times Y \times \R_+) := \Big\{ \eta \in \M(Y \times Y \times R_+) \,\Bigm| \, \int_{Y} s_0^p + s_1^p + S^p d\eta < \infty\Big\}
\end{equation}
satisfying the $p$th homogeneous constraints (with $p=1$ above) given by \[
{\h_i^p\eta = \rho_i\gamma_i}, \quad \Hm^p\eta = \rho\gamma,
\]
where
\[
\h_i^p\, \colon\,\M_p(Y\times Y \times \R_+) \to \M(X), \quad \h_i^p\eta := \pi^{x_i}_{\#}(s_i^p \eta) \in \M(X),
\]
\[
\Hm^p\, \colon\,\M_p(Y\times Y \times \R_+) \to \M(X \times X), \quad \Hm^p\eta := \pi^{(x_0,x_1)}_{\#}(S^p \eta) \in \M(X \times X).
\]
Similarly to the discussion in Section~\ref{sec:UOT-ExtH}, the extended space formulation of ${\rm UOT}_{X,\e}$ will be posed over sets
\begin{align*}
S^p_{\leq}(\mu_0,\mu_1,\nu_X) &:= \{ \eta \in \M(Y \times Y \times \R_+) \, \mid\, \h_i^p \eta \leq \mu_i, \Hm^p \eta \leq \nu_X \},\\
S^p_{=}(\mu_0,\mu_1,\nu_X) &:= \{ \eta \in \M(Y \times Y \times \R_+) \, \mid\, \h_i^p \eta = \mu_i, \Hm^p\eta = \nu_X\},
\end{align*}
where, again clearly, 
\[
S^p_{=}(\mu_0,\mu_1,\nu_X) \subset S^p_{\leq}(\mu_0,\mu_1,\nu_X) \subset \M_p(Y\times Y \times \R_+).
\]
\begin{remark}\label{rem:hom-constaints-relation}
    The reference measure $\nu_X \in \M(X \times X)$ is typically taken to be a product measure
    \[
    \nu_X = \nu_X^0 \otimes v_X^1
    \]
    where $\mu_i \ll \nu_X^i$. For instance, as in \cite{SFVTP19}, we can take $\nu_X = \mu_0 \otimes \mu_1$, then, for ${\eta \in S^p_{=}(\mu_0,\mu_1,\nu_X)}$, it can be readily verified that
    \[
    \pi^i_{\#}(\Hm^p\eta) = \mu_j(X)\h_i^p \eta.
    \]
    More generally, if $\mu_i \ll \nu_X^i$, by the Lebesgue decomposition \eqref{L-decomp} we can write
    \begin{equation}\label{mui-nuxi}
    \mu_i = m_i \nu_X^i, \quad \nu_X^i = n_i \mu_i + \left(\nu_X^i\right)^{\perp},
    \end{equation}
    for some densities $m_i,n_i$ and if additionally $\mu_i \sim \nu_X^i$ then 
    \[
    m_i > 0,\; n_i = \frac{1}{m_i},\; \left(\nu_X^i\right)^{\perp}.
    \]
    An example, fitting with the requirement in \cite{CPSV18} that $\nu_X \in \Pc(X \times X)$, would be to have 
    \[
    \nu_X = \frac{1}{\mu_0(X)\mu_1(X)}\mu_0 \otimes \mu_1,
    \]
    i.e. having constant density $m_i = \mu_i(X)$.
    In the general case \eqref{mui-nuxi}, we note that
    \[
    \eta \in S^p_{=}(\mu_0,\mu_1,\nu_X) \implies \pi^i_{\#}(\Hm^p\eta) = \Bigg(\Big(n_j\h^p_j\eta\Big)(X) + \left(\nu_X^j\right)^{\perp}(X)\Bigg)\Bigg(n_i\h^p_i\eta + \left(\nu_X^i\right)^{\perp}\Bigg). 
    \]

\end{remark}
\begin{theorem}
    The regularised unbalanced optimal transport problem ${\rm UOT}_{X,\e}$, defined in \eqref{eUOT-P}, can be equivalently formulated as 
    \begin{align*}
        {\rm UOT}_{X,\e}(\mu_0,\mu_1) &= \inf_{\eta \in S^p_{\leq}(\mu_0,\mu_1,\nu_X)}\Bigg\{ (H^{\e}_p,\eta) + F(0) \left(\sum_{i=0,1}\left(\mu_i - \h_i^p\eta\right)(X) +\left(\nu_X - \Hm^p\eta\right)(X\times X)\right)\Bigg\},\\
        &=\inf_{\eta \in S^p_{=}(\mu_0,\mu_1,\nu_X)} (H_{p}^{\e},\eta)
    \end{align*}
    where
    \[
    (H^{\e}_p,\eta) := \int_{Y\times Y \times \R_+}H^{\e}_p(x_0,s_0,x_1,s_1,S) d\eta(x_0,s_0,x_1,s_1,S),
    \]
    and 
    \[
    H^{\e}_p(x_0,s_0,x_1,s_1,S) = H_\e(x_0,s_0^p,x_1,s_1^p,S^p).
    \]
\end{theorem}
\begin{proof}
    The result is essentially an immediate consequence of Theorem~\ref{thm:eUOT-X-H} -- the precise argument is verbatim as in the proof of \cite[Theorem~5.8]{LMS18}.
\end{proof}

\subsubsection{Rescalling invariance}\label{sec:eUOT-H-rescale}
The rescalling invariance for the ${\rm UOT}$ problem discussed in Section~\ref{sec:UOT-ExtH-rescale} follows from the fact that the induced marginal perspective cost function $H$ introduced in Section~\ref{sec:UOT-H} satisfies
\[
H\left(x_0,\frac{s_0}{\theta},x_1, \frac{s_1}{\theta}\right) = \frac{1}{\theta}H(x_0,s_0,x_1,s_1). 
\]
Its entropy-regularised counterpart, introduced in Section~\ref{sec:eUOT-H} enjoys a similar property, namely
\[
H_\e\left(x_0,\frac{s_0}{\theta},x_1, \frac{s_1}{\theta},\frac{S}{\theta}\right) = \frac{1}{\theta}H_\e(x_0,s_0,x_1,s_1,S),
\]
which can be readily verified from the explicit form of $H_\e$ established in Lemma~\ref{lem:eH-explicit}. We will now prove the entropy-regularised counterpart to Proposition~\ref{prop:rescale}.
\begin{prop}\label{prop:eUOT_rescale}    Fix $p >0$, $\mu_0,\mu_1 \in \M(X)$, $\nu_X \in \M(X \times X)$ and set
    \[
    s_* := (\mu_0(X)+\mu_1(X) + \nu_X(X \times X))^{\frac 1 p}.
    \]
    Consider functions $\theta\,\colon\,Y \times Y \times \R_+ \to [0,\infty)$ and ${\rm prd}_{\theta}\,\colon\, Y \times Y \times \R_+ \to Y \times Y \times \R_+$ given by
    \begin{equation}
    \theta(x_0,s_0,x_1,s_1,S) = \frac{1}{s_*}(s_0^p + s_1^p + S^p)^{\frac 1 p}, \quad {\rm prd}_{\theta}(x_0,s_0,x_1,s_1,S) := \left(x_0,\frac{s_0}{\theta},x_1,\frac{s_1}{\theta}, \frac{S}{\theta}\right). 
    \end{equation}
    Suppose further that $\eta \in S^p_{\leq}(\mu_0,\mu_1,\nu_X)$ and  define     
    \[
    \tilde \eta := \eta\,\mres\, (Y \times Y \times \R_+) \setminus \{ (x_0,s_0,x_1,s_1,S) \in Y \times Y \times \R_+ \,\mid\, s_0 = s_1 = S = 0, \}
    \]
    and
    \[
    \hat \eta := ({\rm prd}_\theta)_{\#}(\theta^p \tilde \eta).
    \]
    Then
    \begin{itemize}
        \item $(H^{\e}_p,\eta) = (H^{\e}_p,\tilde \eta) = (H^\e_p,\hat \eta)$,
        \item $\hat \eta \in S^p_{\leq}(\mu_0,\mu_1,\nu_X)$ (and if $\eta \in S^p_{=}(\mu_0,\mu_1.\nu_X)$ then $\hat \eta \in S^p_{=}(\mu_0,\mu_1,\nu_X)$),
        \item $\hat \eta \in \Pc(Y \times Y \times \R_+)$,
        \item $\hat \eta ((Y\times Y \times \R_+)[s_*]') = 0$, where
        \begin{alignat*}{2}
        Y[s_*] &:= \{ (x,s) \in Y \mid s \leq s_*\}, \quad &&Y[s_*]' = Y \setminus Y[s_*],\\
        \R_+[s_*] &:= \{ s \in \R_+ \mid s \leq s_*\}, \quad &&\R_+[s_*]' = \R_+ \setminus \R_+[s_*],\\
        (Y\times Y) \times \R_+[s_*] &:= Y[s_*] \times Y[s_*] \times \R_+[s_*], \\ 
        (Y\times Y \times \R_+)[s_*]' &:= Y\times Y \times \R_+ \setminus (Y\times Y \times \R_+)[s_*].
        \end{alignat*} 
    \end{itemize}
\end{prop}
\begin{proof}
    The proof follows from the homogeneity of $H_\e$ as discussed above. {\color{blue}\footnotesize [Expanded exposition of the proof to be added in the next version of the manuscript.]}
\end{proof}
This concludes the {\color{blue} \footnotesize [current version of the]} section on introducing a general framework of working with original-space regularisations of unbalanced optimal transport. We finish this section by discussing the special case of balanced optimal transport, providing counterpart to Section~\ref{sec:OT}.

\subsubsection{Example: the original space based lifting of the entropic regularisation of the balanced optimal transport}\label{sec:eOT-X}
The standard and well theoretically founded (see e.g. \cite{N21}) way of regularising the balanced optimal transport is to consider 
\begin{equation*}
{\rm OT}_{X,\e}(\mu_0,\mu_1) := \inf_{\gamma \in \mathcal{M}(X\times X)} \overline\E_\e(\gamma \mid \mu_0,\mu_1,\nu_X),
\end{equation*}
where  
\begin{equation*}
\overline\E_\e(\gamma \mid \mu_0, \mu_1,\nu_X) := \sum_{i=0,1} \overline{\mathcal{F}}(\gamma_i \mid \mu_i) + (c,\gamma) +\e\F(\gamma \mid \nu_X)
\end{equation*}
where
\begin{align*}
\overline{\mathcal{F}}(\gamma_i \mid \mu_i) &= \int_X \overline F(\sigma_i)d\mu_i + F'_{\infty}\gamma_i^{\perp}(X),\\
{\mathcal{F}}(\gamma \mid \nu_X) &= \int_{X\times X} F(\sigma)d\nu_X + F'_{\infty}\gamma^{\perp}(X\times X).
\end{align*}
All the results of Section~\ref{sec:eUOT_X} apply and hence
\begin{equation}\label{eqn:eOT_H}
    {\rm OT}_{X,\e}(\mu_0,\mu_1) = \inf_{\gamma \in \M(X \times X)} \int_{X \times X}H_\e(x_0,\rho_0(x_0),x_1,\rho_1(x_1),\varrho(x_0,x_1))d\gamma
\end{equation}
where, due to entropy functions $\overline F$ for marginals, we have
\[
H_\e(x_0,s_0,x_1,s_1,S) = \begin{cases}
    ac(x_0,x_1) + \e a R\left(\frac{S}{a}\right),\quad&\text{when } s_0 = s_1 = a,\\
    +\infty,\quad&\text{otherwise.}
\end{cases}
\]
In particular, we note that a necessary condition on the marginals of $\gamma$ for the integral in \eqref{eqn:eOT_H} to be finite is that $a \gamma_i = \mu_i$ for some constant $a > 0$ and so we can again split the infimisation into two steps and observe that
\begin{align*}
    {\rm OT}_{X,\e}(\mu_0,\mu_1) &= \inf_{\gamma \in \Gamma(\mu_0,\mu_1)} \Big\{\inf_{a> 0} \int_{X \times X}\frac{1}{a}\left(c(x_0,x_1) \e R\left(\frac{(1/a) \varrho(x_0,x_1)}{1/a}\right) \right)d (a\gamma) \Big\}\\
    &= \inf_{\gamma \in \Gamma(\mu_0,\mu_1)} \int_{X \times X} c(x_0,x_1) +\e R(\varrho(x_0,x_1))d\gamma,
\end{align*}
thus just reducing to the usual formulation. 

Once we lift to the extended space, similarly to the lifting of the unregularised balanced optimal transport described in Section~\ref{sec:OT}, again such a full reduction does not hold. We have
\[
{\rm OT}_{X,\e}(\mu_0,\mu_1) = \inf_{\alpha \in S^p_=(\mu_0,\mu_1,\nu_X)} \int_{Y \times Y \times \R_+}H_\e (x_0,s_0,x_1,s_1,S)d\alpha.
\]
Clearly a necessary condition on the support of $\alpha$ for the integral to be finite is that 
\[
\alpha((Y\times Y \times \R_+)\setminus \Omega) = 0, \quad \text{where now } \Omega = \{(y_0,y_1,S) \in Y \times Y \times \R_+ \mid s_0 = s_1\}.
\]
In other words, $\alpha$ has to be of the form
\[
\alpha = (x_0,s,x_1,s,S)_{\#}\beta, \quad \beta \in \M(X^2 \times \R_+^2), \quad (x_0,x_1,s,S) \in X^2\times \R_+^2.
\]
Since $\alpha \in S^p_{=}(\mu_0,\mu_1,\nu_X)$, the corresponding $\beta$ has to belong to 
\begin{align*}
\widetilde S^p_{=}(\mu_0,\mu_1,\nu_X) := \Bigg\{ \beta \in &\M(X^2 \times \R_+^2)\, \Bigg|\, \int_{X^2 \times \R_+^2} s^p f(x_i)d\beta = \int_X f(x)d\mu_i,\\ &\int_{X^2 \times \R_+^2} S^pf(x_0,x_1)d\beta = \int_{X\times X}f(x_0,x_1)d\nu_X  \Bigg\}. 
\end{align*}
It readily follows that 
\[
{\rm OT}_\e(\mu_0,\mu_1) = \inf_{\beta \in \widetilde S^p_=(\mu_0,\mu_1,\nu_X)}\int_{X^2 \times \R_+^2} s\, c(x_0,x_1) + \e s R\left(\frac{S}{s}\right)d\beta.
\]
We thus see that for the entropic regularisation of the balanced case, the lifting strategy takes us from $\M(X \times X)$ to $\M(X^2 \times \R_+^2)$.

We now move on to discussing regularisations on the extended space. We will show how the two approaches relate to each other afterwards.

\subsection{Regularisation on the extended space}\label{sec:eUOT_Y}
In this section we consider the entropic regularisation of the extended space formulation of ${\rm UOT}$ described in Section~\ref{sec:UOT-ExtH}. We introduce, for some appropriate reference measure $\nu_Y \in \Pc(Y \times Y)$ (the subscript is kept to distinguish $\nu_Y$ from $\nu_X$ introduced in Section~\ref{sec:eUOT_X}) and small $\e >0$,
\begin{equation}\label{eUOTY-P}
{\rm UOT}_{Y,\e}(\mu_0,\mu_1) := \inf_{\alpha \in \overline{S}^p_{=}(\mu_0,\mu_1)} \Big\{(H_p,\alpha) + \e \F (\alpha \mid \nu_Y) \Big\}.
\end{equation}

We note that, to the best of our knowledge, this approach to regularising unbalanced optimal transport problems remains almost completely unexplored, except for the recent interesting effort in \cite{STVvR23}, where authors study a discrete simplified variant of ${\rm UOT}_{Y,\e}$, discuss the resulting toric geometry of the problem and propose a proof-of-concept numerical algorithm for solving it, akin to the Sinkhorn algorithm. 

We will now proceed to establish several theoretical result about ${\rm UOT}_{Y,\e}$ and begin by quoting the following standard result that will prove useful.
\begin{prop}[\protect{\cite[Section~1]{N21}}]\label{rel-ent-prop}
Let $\nu_Y \in \Pc(Y \times Y)$. The mapping
\[
\Pc(Y \times Y) \ni \alpha  \mapsto \F(\alpha \mid \nu_Y)
\]
is nonnegative, convex and strictly convex on the set where it is finite, with ${\F(\alpha \mid \nu_Y) = 0}$ only if $\alpha = \nu_Y$. Furthermore, the mapping 
\[
\Pc(Y \times Y) \times \Pc(Y \times Y) \ni (\alpha,\nu_Y) \mapsto \F(\alpha \mid \nu_Y)
\]
is jointly convex and jointly lower semicontinuous with respect weak convergence. 
\end{prop}
\begin{remark}\label{nu_Y-restriction}
    In \eqref{eUOTY-P} we restrict the reference measure $\nu_Y$ to lie in $\mathcal{P}(Y \times Y)$ (so in particular to have finite mass). It is possible to extend the framework to $\sigma$-finite measures (\cite{L14,GT19}), i.e. measures for which there exists some measurable function $W\,\colon\, Y \times Y \to \R$ such that 
\[
    z_W := \int_{Y \times Y} \exp(-W) d\nu_Y < +\infty.
\]
We can then the introduce probability measure $\hat\nu_Y := z_W^{-1} e^{-W}\nu_Y \in \mathcal{P}(Y \times Y)$ and hence define the relative entropy with respect to $\nu_Y$ as
\[
\F(Q \mid \nu_Y) := \F(Q \mid \hat \nu_Y) - \int_{Y \times Y} W dQ - \log z_W,
\]
which is well defined for
\[
Q \in \Pc_W(Y \times Y) := \M_W(Y \times Y) \cap \Pc(Y \times Y),
\]
where $\M_W(Y \times Y)$ is given by 
\[
\M_W(Y \times Y) := \{ \alpha \in \M(Y \times Y)\;\mid\; (W,\alpha) < \infty\}.
\]
A specific example to bear in mind is  the volume measure arising from the Riemannian metric on $Y$.
\end{remark}
\begin{remark}
We would particularly like to set $W(y_0,y_1) := s_0^p + s_1^p$, to exploit the connection with spaces $\M_p(Y\times Y)$, as is done in the balanced case e.g. in \cite[Proposition~2.3.]{CDPS17}, but interestingly this choice is incompatible with the volume measure arising from the Riemannian metric on $Y$, as we then get $z_W = \infty$, unless $X$ is compact. {\color{blue}\footnotesize [This point is to be expanded upon in the next version of the manuscript.]}
\end{remark}

Importantly, akin to the final reformulation presented in Theorem~\ref{thm:UOT_as_OT} for the unregularised problem, it is not hard to see that 
\begin{equation}\label{eUOT_Y_as_eOT}
    {\rm UOT}_{Y,\e}(\mu_0,\mu_1) = \inf_{(\beta_0,\beta_1)}\{{\rm OT_\e}(\beta_0,\beta_1)\,\mid\, \beta_i \in \overline{S}^p_=(\mu_i)\}.
\end{equation}
where
\begin{equation}\label{eOT}
    {\rm OT_\e}(\beta_0,\beta_1) := \inf_{\alpha \in \Gamma(\beta_0,\beta_1)} \left\{ (H_p,\alpha) + \e \F(\alpha \mid \nu_Y)\right\}.
\end{equation}
and we recall that 
\[
\Gamma(\beta_0,\beta_1) := \{\alpha \in \Pc(Y\times Y),\; \pi^i_{\#}\alpha = \beta_i\}
\]
is the set of couplings.

We first prove the following.
\begin{prop}\label{eUOT-Y-minimiser}
Under natural assumptions on the cost function $c$ and the reference measure $\nu_Y$, for every $\e >0$, there exists $\overline{\alpha}^\e \in \overline{S}^p_{=}(\mu_0,\mu_1)$ such that
\[
{\rm UOT}_{Y,\e}(\mu_0,\mu_1) = (H_p,\overline{\alpha}^\e) + \e \F(\overline{\alpha}^\e \mid \nu_Y) = {\rm OT}_{\e}(\overline{\alpha}^\e_0,\overline{\alpha}^\e_1),
\]
where $\overline{\alpha}^\e_i := \pi^i_{\#}\overline{\alpha}^\e \in \overline{S}^p_{=}(\mu_i)$ is the $i$th marginal.
\end{prop}
\begin{proof}
The result follows by the direct method of calculus of variations, since both $(H_p,\cdot)$ and  $\F(\cdot \mid \nu_Y)$ are non-negative and lower semicontinuous (c.f. \cite[Section~7.1]{LMS18} and Proposition~\ref{rel-ent-prop}), and we can look for solutions in a narrowly compact subset of $\overline{S}^p_{=}(\mu_0,\mu_1)$ consisting of measures not charging the $(X \times \{0\}) \times (X \times \{0\})$, c.f. \cite[Section~7.3]{LMS18}. The second equality follows from the equivalent formulation of ${\rm UOT}_{Y,\e}$ in \eqref{eUOT_Y_as_eOT}.
\end{proof}

In what follows we also rely on the following standard result. 
\begin{prop}[\protect{\cite[Corollary~5.4]{N21}}]\label{prop:eOT-conv}
For any $\beta_0,\beta_1 \in \Pc(Y)$ the entropy-regularised optimal transport problem ${\rm OT}_{\e}(\alpha_0,\alpha_1)$ introduced in \eqref{eOT} converges to ${\rm OT}(\beta_0,\beta_1)$ introduced Corollary~\ref{thm:UOT_as_OT}, that is
\[
\lim_{\e \to 0} {\rm OT}_{\e}(\beta_0,\beta_1) = {\rm OT}(\beta_0,\beta_1).
\]
\end{prop}

We will now prove that the extended-space entropy-regularised problem does indeed approximate the unbalanced optimal transport problem.
\begin{theorem}\label{thm:eUOT-Y_to_UOT_Y}
For any $\mu_0,\mu_1 \in \mathcal{M}(X)$ the $Y$-space regularised unbalanced optimal transport problem ${\rm UOT}_{Y,\e}(\mu_0,\mu_1)$ introduced in \eqref{eUOTY-P} converges to ${\rm UOT}(\mu_0,\mu_1)$ introduced in \eqref{UOT-P}, that is
\[
\lim_{\e \to 0} {\rm UOT}_{Y,\e}(\mu_0,\mu_1) = {\rm UOT}(\mu_0,\mu_1).
\]
\end{theorem}
\begin{proof}
Proposition~\ref{prop:alpha_bar} ensures existence of $\overline{\alpha} \in \overline{S}^p_{=}(\mu_0,\mu_1)$ such that
\[
{\rm UOT}(\mu_0,\mu_1) = (H_p,\overline{\alpha})= {\rm OT}(\overline{\alpha}_0,\overline{\alpha}_1),
\]
where $\overline{\alpha}_i := \pi^i_{\#}\overline\alpha$ is the $i$th marginal.

By Proposition~\ref{prop:eOT-conv}, it holds that 
\[
\lim_{\e \to 0} {\rm OT}_{\e}(\overline{\alpha}_0,\overline{\alpha}_1) = {\rm OT}(\overline{\alpha}_0,\overline{\alpha}_1) 
\]
which implies that, as $\e \to 0$,
\[
{\rm UOT}_{Y,\e}(\mu_0,\mu_1) \leq {\rm OT}_{\e}(\overline{\alpha}_0,\overline{\alpha}_1) \to {\rm OT}(\overline{\alpha}_0,\overline{\alpha}_1) = {\rm UOT}(\mu_0,\mu_1)
\]
and hence
\[
\lim_{\e \to 0} {\rm UOT}_{Y,\e}(\mu_0,\mu_1) \leq {\rm UOT}(\mu_0,\mu_1).
\]
To conclude the result it would suffice to show that as $\e$ decreases, the values ${\rm UOT}_{Y,\e}(\mu_0,\mu_1)$ are monotone decreasing, as this would imply 
\[
\lim_{\e \to 0} {\rm UOT}_{Y,\e}(\mu_0,\mu_1) \geq {\rm UOT}(\mu_0,\mu_1),
\]
as discussed in \cite[Section~5.1]{N21}.

To show it we rely on Proposition~\ref{eUOT-Y-minimiser}, which established that $\exists\; \overline{\alpha}^{\e} \in \overline{S}^p_{=}(\mu_0,\mu_1)$ such that 
\[
{\rm UOT}_{Y,\e}(\mu_0,\mu_1) = (H_p,\overline{\alpha}^\e) + \e \F(\overline{\alpha}^\e \mid \nu_Y)= {\rm OT}_{\e}(\overline{\alpha}^\e_0,\overline{\alpha}^\e_1).
\]
Take $\e_2 > \e_1 > 0$ and denote by $\overline{\alpha}^i = \overline{\alpha}^{\e_i}$. We then have 
\begin{align*}
{\rm UOT}_{Y,\e_1}(\mu_0,\mu_1) &\leq (H_p, \overline{\alpha}^2) + \e_1 \F(\overline{\alpha}^2 \mid \nu_Y) = {\rm UOT}_{Y,\e_2}(\mu_0,\mu_1) + (\e_1 - \e_2)\F(\overline{\alpha}^2 \mid \nu_Y)\\
&\leq {\rm UOT}_{Y,\e_2}(\mu_0,\mu_1), 
\end{align*}
where the last equality follows from $\e_2 > \e_1$ and $\F(Q \mid \nu_Y) \geq 0$ (c.f. Proposition~\ref{rel-ent-prop}).
\end{proof}
\begin{remark}
Reflecting the fact that the homogeneous marginal mapping \linebreak ${\h_i^p\,\colon \M_2(Y \times Y) \to \M(X)}$ is not injective, instead of a-priori fixing a reference measure $\nu_X \in \Pc(Y \times Y)$, it is also feasible to introduce the entropy regularisation as  
\begin{equation}
    \widehat{\rm UOT}_{Y,\e}(\mu_0,\mu_1) = \inf_{(\beta_0,\beta_1)}\{\widehat{\rm OT}_{\e}(\beta_0,\beta_1)\,\mid\, \beta_i \in S^p_{=}(\mu_i)\},
\end{equation}
where
\begin{equation}
    \widehat{\rm OT}_{\e}(\beta_0,\beta_1) := \inf_{\alpha \in \Gamma(\beta_0,\beta_1)} \left\{ (H_p,\alpha) + \e \F(\alpha \mid \alpha_0 \otimes \alpha_1)\right\},
\end{equation}
with equivalent formulation given by
\[
\widehat{\rm UOT}_{Y,\e}(\mu_0,\mu_1) = \inf_{\alpha \in \overline{S}^p_{=}(\mu_0,\mu_1)}\big\{ (H_p,\alpha) + \e \F(\alpha \mid \alpha_0 \otimes \alpha_1)\big\}.
\]
The joint lower semicontinuity of $(\alpha,\nu_Y) \mapsto \F(\alpha,\nu_Y)$, see Proposition~\ref{rel-ent-prop}, ensures that a convergence result akin to Theorem~\ref{thm:eUOT-Y_to_UOT_Y} still holds in this case. {\color{blue}\footnotesize [This point is to be expanded upon in the next version of the manuscript.]} 
\end{remark}

{\color{blue}\footnotesize [The theory of the extended space regularisation is to be significantly expanded upon in the next version of the manuscript, based on ideas developed in Section~\ref{sec:eUOT-unify}.]}

\subsubsection{Example: the extended space based lifting of the entropic regularisation of the balanced optimal transport}\label{sec:eOT-Y}
Starting from \eqref{eqn:OT-H-ext}, we can consider {\color{blue} \footnotesize [This section will be added in the next version of the manuscript and relies on ideas developed in Section~\ref{sec:eUOT-unify}.]}

\section{Discussion on deriving a unified extended framework for comparing $X$ space regularisation and $Y$ space regularisation}\label{sec:eUOT-unify}
In this section we report on some preliminary observations related to deriving a common framework in which both the ${\rm UOT}_{X,\e}$ problem, discussed in Section~\ref{sec:eUOT_X}, and the ${\rm UOT}_{Y,\e}$ problem, discussed in Section ~\ref{sec:eUOT_Y}, can be directly compared. This is particularly important because we view deriving such a direct correspondence, firstly, as key to proving convergence of values of 
\[
{\rm UOT}_{X,\e}(\mu_0,\mu_1) \to {\rm UOT}(\mu_0,\mu_1)\,\text{ as }\,\e \to 0,
\]
and, secondly, as a way of developing a fuller theory of entropic regularisations of unbalanced optimal transport problems. Ultimately we also hope to tie the static theory to the different possible dynamic formulations of entropy-regularised unbalanced optimal transport problems (see Section~\ref{sec:eUOT-dynamic} for a further dicsussion on this).

The basic idea behind deriving a common extended framework stems from the dimensional discrepancy between extended descriptions: in the ${\rm UOT}_{X,\e}$ problem, as discussed in Section~\ref{sec:eUOT-ExtH}, the extended formulation seeks minimisers among a subset of measures in 
\[
\M(Y \times Y \times \R_+) \equiv \M(X^2 \times \R_+^3),
\]
whereas in the ${\rm UOT}_{Y,\e}$ problem, as discussed in Section~\ref{sec:eUOT_Y}, the minimisers are sought among measures in 
\[
\M(Y \times Y) \equiv \M(X^2 \times R_+^2).
\]

In our view, a promising approach to deriving a unified framework, rests on the idea of subsequent liftings, so that we can reformulate both problems as two minimisation problems over a common space of measures $\M(X^2 \times \R_+^4)$,. This endeavour proves to be surprisingly tricky and we set our strategy first by showing how to lift the unbalanced optimal transport problem to a minimisation problem on $\M(X^2 \times \R_+^3)$, followed by an outline of the strategy for original space and extended space regularisations. {\color{blue}[\footnotesize This section is to be substantially expanded in the next version of the manuscript or delayed to a follow-up paper.]} 

\subsection{Further lifting of the ${\rm UOT}$ problem}\label{sec:UOT-H-ext2}
Starting from Theorem~\ref{thm:UOT-H-ext}, we observe that
\begin{equation}\label{eqn:UOT-H-ext2_1}
{\rm UOT}(\mu_0,\mu_1) = \inf_{\alpha \in S^p_=(\mu_0,\mu_1)} (H_p,\alpha) = \inf_{\alpha \in M(Y \times Y)} \sum_i \overline\F(h_i^p\alpha \mid \mu_i) + (H_p,\alpha),
\end{equation}
where we recall that the definition of $\overline{\F}$ was introduced in Section~\ref{sec:OT}. This formulation can be seen as a primal formulation of a balanced variant of what we term a \emph{second order unbalanced optimal transport problem}. Using the steps detailed in Section~\ref{sec:UOT-equiv} and in Section~\ref{sec:eUOT_X}, we can, starting with \eqref{eqn:UOT-H-ext2_1}, derive the reverse formulation, the dual formulation, the homogeneous formulation and the extended space formulation.

The reverse formulation can be readily shown to be given by  
\[
{\rm UOT}(\mu_0,\mu_1) =  \inf_{\alpha \in M(Y \times Y)} \int_{Y \times Y} \left(\sum_i s_i^p \overline R(\rho_i(x_i)) + H_p\right)d\alpha,
\]
where, we employ the Lebesgue decomposition between $\h_i^p\alpha$ and $\mu_i$, that is 
\[
\h_i^p\alpha = \sigma_i \mu_i + (\h_i^p\alpha)^{\perp}, \quad \mu_i = \rho_i \h_i^p\alpha + \mu_i^\perp,
\]
and in the present case it is clear that the singular parts are effectively null. Furthermore, in the present case, clearly, $\overline{R}(s) = s\overline{F}\left(\frac{1}{s}\right) = \overline{F}(s)$.

Likewise, relying on Theorem~\ref{thm:FRT}, and noting that 
\[
(\overline F)^*(\phi) = (\overline R)^*(\phi) = \phi,
\]
it is not hard to show that the dual formulation is given by 
\[
{\rm UOT}(\mu_0,\mu_1) = \sup_{\phi \in \Phi_Y} \sum_i \mu_i(\phi_i),
\]
where, crucially,
\[
\Phi_Y := \{ \phi = (\phi_0,\phi_1) \in C_b(X) \times C_b(X) \mid (s_0^p \phi_0) \oplus (s_1^p \phi_1) \leq H\}. 
\]
We can thus introduce the \emph{second order marginal perspective function} 
\[
\tilde H \,\colon\, X^2 \times \R_+^3 \to \R,  
\]
given by, suppresing the dependence $\tilde H \equiv \tilde H (x_0,x_1,s_0,s_1,w_0,w_1)$, 
\begin{align*}
\tilde H &:= \inf_{t>0} \Bigg\{t\Bigg(s_0^p\overline{R}\left(\frac{w_0}{t}\right) + s_1^p\overline{R}\left(\frac{w_1}{t}\right) + H(x_0,s_0,x_1,s_1)\Bigg)\Bigg\}\\
&= \inf_{t>0} \Bigg\{ w_0s_0^p \overline{F}\left(\frac{t}{w_0}\right) + w_1s_1^p \overline{F}\left(\frac{t}{w_1}\right) + tH(x_0,s_0,x_1,s_1) \Bigg\}\\
&= \sup_{(\phi_0,\phi_1) \in \R_+^2}\Bigg\{ -s_0^pw_0(\overline F)^*(-\phi_0) - s_1^pw_1 (\overline F)^*(-\phi_1)\, \Big|\, s_0^p\phi_0 + s_1^p\phi_1 \leq H(x_0,s_0,x_1,s_1)  \Bigg\},
\end{align*}
with the latter two equalities following from obvious adjustment of arguments presented throughout this paper, in Section~\ref{sec:UOT-equiv} and in Section~\ref{sec:eUOT_X}.

As a result, it can be shown that the homogenous formulation is 
\[
{\rm UOT}(\mu_0,\mu_1) = \inf_{\alpha \in M(Y \times Y)} \int_{Y \times Y} \tilde H(x_0,x_1,s_0,s_1,\rho_0(x_0),\rho_1(x_1)) d\alpha.
\]

As in Section~\ref{sec:OT}, the entropy functionals $\overline F$ ensure that 
\[
\tilde H(x_0,x_1,s_0,s_1,w_0,w_1) = \begin{cases}
    aH(x_0,s_0,x_1,s_1),\quad&\text{when } w_0 = w_1 = a,\\
    +\infty,\quad&\text{otherwise.}
\end{cases}
\]
In particular, a necessary condition on the homogeneous marginals of $\alpha$ for the integral to be finite is that  $a \h_i^p\alpha = \mu_i$ for some constant $a > 0$.

The second order extended formulation is given by
\[
{\rm UOT}(\mu_0,\mu_1) = \inf_{\eta \in \mathcal S^p_=(\mu_0,\mu_1)} \int_{X^2\times \R_+^4} \tilde H(x_0,x_1,s_0,s_1,w_0,w_1) d\eta,
\]
where 
\[
\mathcal S^p_{=}(\mu_0,\mu_1) := \Big\{ \eta \in M(X^2 \times \R_+^4) \mid \pi^{x_i}_{\#}(s_i^pw_i \eta) = \mu_i \Big\}.
\]
In particular, for any $\alpha \in M(Y \times Y)$ such that $\mu_i = \rho_i \h_i^p\alpha$ (here the setup effectively only permits constant density as discussed above), we have that 
\[
(x_0,s_0,\rho_0(x_0),x_0,s_0,\rho_1(x_1))_{\#}\alpha \in \mathcal S^p_{=}(\mu_0,\mu_1).
\]

Finally, as was also the case in the example of lifting the balanced optimal transport problem described in Section~\ref{sec:OT},
a necessary condition on the support of $\eta$ for the integral to be finite is that 
\[
\eta((X^2 \times \R_+^4)\setminus \Omega) = 0, \quad \text{where } \Omega = \{(x_0,x_1,s_0,s_1,w_0,w_1) \in X^2 \times \R_+^4 \mid w_0 = w_1\}.
\]
In other words, $\eta$ has to be of the form
\[
\eta = (x_0,x_1,s_0,s_1,w,w)_{\#}\xi, \quad \xi \in \M(X^2 \times \R_+^3),\; (x_0,x_1,s_0,s_1,w) \in X^2\times \R_+^3.
\]
Since $\eta \in \mathcal{S}^p_{=}(\mu_0,\mu_1)$, the corresponding $\xi$ has to belong to
\[
\widetilde{\mathcal S}^p_{=}(\mu_0,\mu_1) := \left\{ \xi \in \M(X^2\times \R^3_+)\, \Big|\, \pi^{x_i}_{\#}(s^pw_i\xi) = \mu_i \right\},
\]
where, for any measurable function $f\,\colon\,X \to \R$,
\[
\int_X f(x) d\pi^{x_i}_{\#}(s^pw_i\xi) = \int_{X^2 \times \R_+^3} s^p w_if(x_i)d\xi.
\]
It readily follows that 
\[
{\rm UOT}(\mu_0,\mu_1) = \inf_{\xi \in \widetilde{\mathcal S}^p_=(\mu_0,\mu_1)}\int_{X^2 \times \R^3_+} w\, H(x_0,s_0,x_1,s_1)d\xi,
\]
which is a minimisation problem over a subset of the space $\M(X^2 \times \R_+^3)$, which is what we set out to outline. 
\begin{remark}\label{rem:hoUOT}
    A full \emph{second order unbalanced optimal transport problem}, to be studied in the future, concerns relaxing the homogenenous marginal constraints, that is replacing $\overline \F$ in \eqref{eqn:UOT-H-ext2_1} with a general entropy functional $\F$, e.g. a one which uses the KL divergence as the entropy function.  
\end{remark}
\subsection{Lifting of the ${\rm UOT}_{Y,\e}$ problem to $\M(X^2 \times \R_+^4)$} $\,$ {\color{blue} \footnotesize [This section is to be added in the next version of the manuscript. The basic observation is that the construction outlined above works, with obvious adjustments, for ${\rm UOT}_{Y,\e}$ too.] }

\subsection{Lifting of the ${\rm UOT}_{X,\e}$ problem to $\M(X^2 \times \R_+^4)$} $\,$ To conduct a similar lifting approach in the case of the ${\rm UOT}_{X,\e}$ problem, we have to proceed in two steps. We first do the first order lifting, as described in Section~\ref{sec:eUOT-ExtH}, which lets us formulate the problem as a minimisation over a subset of the space $\M(X^2 \times \R_+^3)$, namely 
\begin{align}
{\rm UOT}_{X,\e}(\mu_0,\mu_1) &= \inf_{\eta \in S^p_=(\mu_0,\mu_1,\nu_X)} (H^{\e}_p,\alpha) \nonumber\\
&= \inf_{\eta \in M(X^2 \times \R_+^3)} \sum_i \overline\F(\h_i^p\eta \mid \mu_i) + \overline\F(\Hm^p\eta \mid \nu_X) +  (H^\e_p,\alpha), \label{eqn:eUOT-X-H-ext2_2}
\end{align}

Using the steps detailed in Section~\ref{sec:UOT-H-ext2}, we can then easily derive an extended formulation where the problem is posed as a minimisation over a subset of measures in $\M(X^2 \times \R_+^6)$, but the sharpness of the entropy functionals $\overline \F$ ensures that we can reduce the problem to a minimisation over a subset of $\M(X^2 \times \R_+^4)$. {\color{blue} \footnotesize [This section is to be considerably expanded in the next version of the manuscript.]}

\section{Discussion on the dynamic formulation}\label{sec:eUOT-dynamic}
So far we have exclusively focused on the static formulation of unbalanced optimal transport problems and the two fundamentally different approaches in which they can be entropy-regularised. In the subsequent heuristic discussion, we will attempt to elucidate how a similar distinction can be made starting from the dynamic formulation \cite{BB00}, \cite[Section~8]{LMS18} and how in this case it appears natural to view the extended space regularisation as superior, given the overall setup of unbalanced optimal transportation problems.

We first very briefly recall the dynamic formulation of ${\rm UOT}(\mu_0,\mu_1)$, followed by a discussion on the possible entropic regularisations.
\subsection{Dynamic formulation of UOT}
The unbalanced optimal transport problem ${\rm UOT}(\mu_0,\mu_1)$ introduced in Section~\ref{sec:UOT-P}, admits another reformulation, typically referred to as \emph{dynamic}. 
\begin{prop}[\protect{\cite[Theorem~8.18, Theorem~8.19]{LMS18}}]\label{prop:eUOT-dynamic}
For a specific choice of cost function and the marginal entropy functions, the ${\rm UOT}$ problem in \eqref{UOT-P} admits a dynamic formulation, namely \begin{equation}\label{UOT-dynamic}
\mathrm{UOT}(\mu_0,\mu_1):= \inf_{(\rho, v, r) \in {\rm Adm}(\mu_0,\mu_1)}\left\{\int \int_0^1 (|v|^2 + |r|^2) \rho\,dxdt  \right\}
\end{equation}
where, for $(\rho,v,r)$ belonging to appropriately defined function spaces (in particular, $v$ is a vector field),
\[
{\rm Adm}(\mu_0,\mu_1) :=\{ \rho, v, r \,\mid\,  \partial_t \rho + \div(\rho v) = \rho r,\quad  \rho|_{t=i} = \mu_i\}.
\]
In particular, it follows that in the triplets $(\rho,v,r)$ minimising \eqref{UOT-dynamic}, the vector field $v$ is the gradient of the scalar field $r$, that is $v = \nabla r$.
\end{prop}

\subsection{Dynamic formulations of entropy regularised unbalanced optimal transport problems}
The key observation that led us to this work comes from the reformulation of the unbalanced optimal transport problem an optimal transport problem on the extended space, see Corollary~\ref{thm:UOT_as_OT} -- since in this formulation we allow transport across the radial component, which in the original problem corresponds to the creation/annihilation of mass, it would be a natural requirement for the entropic regularisation to also handle both transport and creation/annihilation of mass (reaction). In the context of the dynamical formulation, this would correspond to the entropy-regularised continuity equation to contain entropic reaction terms in addition to diffusion terms.

This, however, is in contrast with known dynamic formulations of entropy-regularised unbalanced optimal transport, see e.g. \cite{BL21}, which is
\begin{equation}\label{eUOT-dynamic1}
\mathrm{UOT}^{\rm D}_{X,\e}(\mu_0,\mu_1):= \inf_{(\rho, v, r) \in {\rm Adm}_{X,\e}(\mu_0,\mu_1)}\left\{\int \int_0^1 (|v|^2 + |r|^2) \rho\,dxdt  \right\}
\end{equation}
where, for $(\rho,v,r)$ belonging to appropriate functions spaces, 
\[
{\rm Adm}_{X,\e}(\mu_0,\mu_1) :=\Big\{ \rho, v, r  \,\mid\,  \partial_t \rho + \div(\rho v) = \rho r + \frac{\e}{2}\Delta \rho,\quad  \rho|_{t=i} = \mu_i\Big\}.
\]
In this formulation all the $\e$-terms are associated with diffusion and not reaction and this formulation is reminiscent of the static entropy regularisation on the original space, ${\rm UOT}_{X,\e}(\mu_0,\mu_1)$ discussed in Section~\ref{sec:eUOT_X}, hence the subscript $X$. 

We will employ an argument similar to the one discussed in \cite{LLW19} to develop higher-order (Fisher information) regularisation for Wasserstein gradient flows, to outline how this way of regularising can be can viewed as leading to first order errors. Starting form the continuity equation associated with ${\rm Adm}_{X,\e}(\mu_0,\mu_1)$,
\[
\partial_t \rho + \div(\rho v) = \rho r + \frac{\e}{2}\Delta \rho
\]
and exploiting that $\Delta \rho = \div (\rho \nabla \log \rho)$, we define
\begin{equation}\label{eqn:rho_v_r_sub}
\tilde v := v - \frac{\e}{2} \nabla \log \rho, \quad \tilde r = r - \e s,
\end{equation}
where the function $s$ is to be determined. The continuity equation now becomes
\begin{equation}\label{new-ce}
\partial_t \rho + \div(\rho \tilde v) = \rho (\tilde r + \e s).
\end{equation}
It follows that, for any triplet $(\rho,v,r) \in {\rm Adm}_{X,\e}(\mu_0,\mu_1)$ and with substitution \eqref{eqn:rho_v_r_sub}, the integral in \eqref{eUOT-dynamic1} is given by 
\begin{align}
\int_0^1 \int_X \left(|v|^2 + |r|^2\right){\rm d}\rho {\rm d}t &=\int_0^1 \int_X \left(|\tilde v|^2 + \frac{\e^2}{4}|\nabla \log \rho|^2 + \e \tilde v \cdot \nabla \log \rho\right){\rm d}\rho {\rm d}t \label{a}\\ &+ \int_0^1 \int_X \left(|\tilde r|^2 + \e^2|s|^2  + 2\e \tilde r s\right){\rm d}\rho {\rm d}t. \nonumber
\end{align}
As in \cite{LLW19}, to handle the term $\int_0^1\int_X \e \tilde v \cdot \nabla \log \rho{\rm d}\rho {\rm d}t$, we define ${H[\rho]:= \int_X \rho \log \rho}dx$ and observe that, since $\rho$ satisfies \eqref{new-ce},
\begin{align*}
\frac{{\rm d}}{{\rm d}t}H[\rho] &= \int (\log \rho +1) \partial_t \rho dx = \int (\log \rho +1)(-\div(\rho \tilde v) + \rho(\tilde r + \e s))dx\\
&= \int (\nabla \log \rho \cdot \tilde v)\rho dx + \int (\log \rho +1)(\tilde r + \e s)\rho dx,
\end{align*}
where we have used \eqref{new-ce} in the second equality and integration by parts in the third. It thus follows that 
\begin{align*}
\e \int_0^1\int_X \tilde v \cdot \nabla \log \rho{\rm d}\rho {\rm d}t &= \e \int_0^1 \frac{{\rm d}}{{\rm d}t}H[\rho] {\rm d}t \\ &- \e\int_0^1 \int_X (\log \rho +1) \tilde r \rho - \e^2 \int_0^1 \int_X(\log \rho +1) s \rho.
\end{align*}

Recalling that the function $s$ was left to be determined, it appears natural to set 
\[
s = \frac12 (\log \rho +1),
\]
as this leads to a cancellation of some of the linear terms in $\e$, namely \eqref{a} becomes
\begin{align*}
\int_0^1 \int_X(|v|^2 + |r|^2){\rm d}\rho {\rm d}t = \int_0^1 \int_X \Bigg(|&\tilde v|^2 + |\tilde r|^2 + \frac{\e^2}{4} \left(|\nabla \log \rho|^2 +  |\log \rho + 1|^2  \right)\Bigg){\rm d}\rho {\rm d}t \\
&+ \e\left(H[\mu_1] - H[\mu_0]\right)
\end{align*}

To us this is a very clear indication that, reversing the substitution \eqref{eqn:rho_v_r_sub} in \eqref{new-ce}, a more accurate continuity equation is 
\begin{equation}\label{CE-best}
\partial_t \rho + \div(\rho v) = \rho r + \frac{\e}{2}\div(\rho\nabla \log \rho) - \frac{\e}{2}\rho(\log \rho +1).
\end{equation}

This should not be surprising in the light of the second part of Proposition~\ref{prop:eUOT-dynamic}, establishing that in the unregularised formulation, the minimising triplet satisfies ${v = \nabla r}$. This is  preserved when regularising with \eqref{CE-best}, as when we set 
\[
\tilde v = v - \frac{\e}{2} \nabla \log \rho, \quad \tilde r = r - \frac{\e}{2}(\log \rho +1),
\]
we obtain the standard continuity equation 
\[
\partial_t \rho + \div(\rho \tilde v) = \rho \tilde r
\]
and importantly then $\tilde v = \nabla \tilde \xi$ where $\tilde \xi = \xi - \frac{\e}{2}(\log \rho + C)$ for any constant $C$ and likewise 
$\tilde r = \tilde \xi$ for $C=1$.

To formalise the above discussion, we introduce 
\begin{equation}\label{eUOT-dynamic-Y}
\mathrm{UOT}^{\rm D}_{Y,\e}(\mu_0,\mu_1):= \inf_{(\rho, v, r) \in {\rm Adm}_{Y,\e}(\mu_0,\mu_1)}\left\{\int \int_0^1 (|v|^2 + |r|^2) \rho\,dxdt  \right\}
\end{equation}
where, for $(\rho,v,r)$ belonging to appropriate functions spaces, 
\[
{\rm Adm}_{Y,\e}(\mu_0,\mu_1) :=\{ \rho, v, r  \,\mid\,  \partial_t \rho + \div(\rho v) - \frac{\e}{2}(\rho\nabla \log \rho) = \rho r - \frac{\e}{2}(\log \rho +1) ,\; \rho|_{t=i} = \mu_i\}.
\]

We employ the subscript $Y$ because now the entropic terms enter both in the diffusion and reaction terms, reminiscent of the extended space static regularisation discussed in Section~\ref{sec:eUOT_Y}.

From the above discussion we conjecture that the ${\rm UOT}_{Y,\e}(\mu_0,\mu_1)$ is a higher order approximation scheme to ${\rm UOT}(\mu_0,\mu_1)$, by which we mean that the convergence, as $\e \to 0$, of 
\[
{\rm UOT}^{\rm D}_{Y,\e}(\mu_0,\mu_1) \to {\rm UOT}(\mu_0,\mu_1)
\]
is asymptotically more rapid than of
\[
{\rm UOT}^{\rm D}_{X,\e}(\mu_0,\mu_1) \to {\rm UOT}(\mu_0,\mu_1).
\]
It is our longer term goal to establish such results and also derive direct passages between corresponding static and dynamic formulations of entropy-regularised unbalanced optimal transport problems. For the balanced case, this is done in the Appendix~\ref{app1}.
\section{Outlook}
The work presented in this manuscript develops several aspects of a theory of entropy regularisation of unbalanced transport problems and importantly opens up many avenues of further theoretical work, which we will now briefly discuss. 

\subsection*{Higher order lifting theory} As hinted at in Remark~\ref{rem:hoUOT}, we can plausiably imagine formulating and studying \emph{higher order unbalanced optimal transport problems} of the form
\[
{\rm HUOT}(\mu_0,\mu_1) = \inf_{\alpha \in M(Y \times Y)} \sum_i \F(h_i^p\alpha \mid \mu_i) + (H_p,\alpha),
\]
where we relax the homogeneous marginal constraint and allow for $\h_i^p\alpha$ to deviate from $\mu_i$. This can also be done for its entropy regularisations and one particular feasible scenario in which this can be useful concerns establishing convergence proofs for ${\rm UOT}_{X,\e}(\mu_0,\mu_1)$ -- perhaps after a higher order lifting we have to relax the constraints, as otherwise the problem is not well-posed? 

More broadly, should such a second order problem (and/or its entropic counterpart) prove feasible and interesting, one can also iterate this procedure and look at $n$th order problems too. 

We are not aware of any physical justification for posing such higher order problems, but perhaps a link can be established in relation to dynamic Schrödinger problems, with the entropy minimisation taken with respect to some stochastic process, in the light of the recent interesting work on branching Brownian motion \cite{BL21}.

\subsection*{Passages between different regularisations} An obvious next step to develop a full theory is to be able to derive explicit connections, either in terms formulae or asymptotic results, between static and dynamic formulations of original- and extended-space regularisations. In particular, to achieve the passage between static and dynamic formulation, we hope to leverage the rich literature on such conversions in the balanced case (\cite{GLR17,GT19} and see Appendix~\ref{app1}).

\subsection*{Convergence and convergence rate proofs} We aim to establish the convergence result
\[
{\rm UOT}_{X,\e}(\mu_0,\mu_1) \to {\rm UOT}(\mu_0,\mu_1), \text{ as } \e \to 0.
\]
We hope that the higher order lifting strategy will ensure that we can leverage the corresponding result for ${\rm UOT}_{Y,\e}(\mu_0,\mu_1)$ to achieve that. 

Based on the discussion in Section~\ref{sec:eUOT-dynamic}, it is also our hope to establish some results in relation to the convergence rate in $\e$ of the two types regularisations.

\subsection*{Novel numerical algorithms}
As already noted, authors in \cite{STVvR23} propose numerical algorithms based on the idea of extended space regularisations and e.g. leveraging the generalized iterative scaling \cite{DR72}. We hope that the current effort on putting the theory of entropic regularisations of unbalanced optimal transport on a more rigorous mathematical footing will result in further developments, e.g. leveraging the ideas of higher order liftings. We also would like to point that already in the balanced case, as discussed in Section~\ref{sec:OT} and as will be detailed in Section~\ref{sec:eOT-Y}, one can think numerical algorithms approximating the balanced optimal transport problems on the extended space. 
\appendix 
\section{Passage between static and dynamic  formulations of entropy regularised balanced optimal transport}\label{app1}

The original space based entropy regularisations of optimal transport problems (see Section~\ref{sec:eOT-X})  have been formulated in various closely related ways in the literature, using different notations. Let us introduce them now using one unified notation. We set $X = \R^d$ and suppress the $X$ subscript to make the notation marginally less cumbersome. We consider three static formulations
\begin{align}
{\rm OT}_{\e,1}(\mu,\nu) &:= \inf_{\gamma \in \Pi(\mu,\nu)}\{ (c,\gamma) + \e H(\gamma) \}\\
{\rm OT}_{\e,2}(\mu,\nu) &:= \inf_{\gamma \in \Pi(\mu,\nu)}\{ (c,\gamma) + \e H(\gamma\mid \mu \otimes \nu ) \}\\
{\rm OT}_{\e,3}(\mu,\nu) &:= \inf_{\gamma \in \Pi(\mu,\nu)}\{ \e H(\gamma \mid  K ) \}
\end{align}
where $H(\cdot \mid \cdot)$ is the usual definition of relative entropy and $H(\cdot) = H(\cdot\mid \mathcal{L})$ is the relative entropy with respect the the Lebesgue measure, the pairing is defined as 
\[
(c,\gamma) := \int c(x,y)d\gamma(x,y),\quad c(x,y) = |x-y|^2,
\]
and the measure $K$ corresponds to the heat kernel evaluated at time $t=\e / 2$ and is defined as
\begin{align*}
dK(x,y) &= (2\pi \e)^{-d/2} \exp\left(\frac{-c(x,y)}{2\e}\right)dxdy.
\end{align*}

The form $W_{\e,1}$ is as in \cite{CDPS17}, $W_{\e,2}$ is as in \cite{N21} and also \cite{CRLVP20} (but replacing $\e$ with with $\lambda = \e / 2$ and this prefactor will turn out crucial when reconciling all these formulations). The form $W_{\e,3}$ is used in \cite{GT19} to establish the explicit link between static and dynamic formulations. 

Direct calculations reveal that
\begin{align*}
W_{\e,2}^2(\mu,\nu) &= W_{\e,1}^2(\mu,\nu) - \e\left( H(\mu) + H(\nu)\right)\\
W_{\e,3}^2(\mu,\nu) &= \frac12 W_{2\e,1}^2(\mu,\nu) + \frac{d\e}{2} \log(2 \pi \e),
\end{align*}
so we have a clear passage between all variants (note the change from $\e$ to $2\e$ in the subscript in the second equality).

As far as dynamic formulations are concerned, two variants are typically used, namely
\begin{alignat*}{2}
\widetilde{W}_{\e,a}^2(\mu,\nu) &:= \inf_{\rho,v}\Big\{ \int_0^1 &&\int |v|^2 d\rho(x)dt\mid \partial_t \rho + {\rm div}(\rho v) = \frac{\e}{2} \Delta \rho,\; \rho_0 = \mu,\; \rho_1 = \nu\Big\}\\
\widetilde{W}_{\e,b}^2(\mu,\nu) &:= \inf_{\rho,v}\Big\{ \int_0^1 &&\int \Big(|v|^2+\frac{\e^2}{4}|\nabla \log \rho|^2\Big) d\rho(x)dt\\
&\, &&\partial_t \rho + {\rm div}(\rho v) = 0,\; \rho_0 = \mu,\; \rho_1 = \nu\Big\}
\end{alignat*}
and a direct calculation as in \cite{LLW19} reveals
\[
\widetilde{W}_{\e,a}^2(\mu,\nu) = \widetilde{W}_{\e,b}^2(\mu,\nu) + \e\left(H(\nu) - H(\mu)\right).
\]

It follows from \cite[Proposition 3.6]{GT19} that
\begin{align*}
W_{\e,3}^2(\mu,\nu) &= \frac12 \widetilde{W}_{\e,b}^2(\mu,\nu) + \frac{\e}{2}\left(H(\mu) + H(\nu)\right)\\
&= \frac12 \widetilde{W}_{\e,a}^2(\mu,\nu) + {\e}H(\mu)
\end{align*}

The formula between two most widely used variants is thus
\[
W^2_{\e,1}(\mu,\nu) = \widetilde{W}_{\frac{\e}{2},a}^2 (\mu,\nu) + \e H(\mu) - \frac{d\e}{2}\log(\pi \e)
\]

This is consistent with \cite{CRLVP20}.\\

\medskip

\printbibliography

\end{document}